\documentclass[12pt]{amsart}

\usepackage{amsfonts,amsmath,amsthm,amssymb}
\usepackage[mathscr]{eucal}
\usepackage{eufrak}

\usepackage{enumitem}
\setlist[enumerate]{font=\normalfont}

\usepackage{hyperref}

\newtheorem{thm}{Theorem}[section]
\newtheorem{cor}[thm]{Corollary}
\newtheorem{lm}[thm]{Lemma}
\newtheorem{prp}[thm]{Proposition}

\theoremstyle{definition}
\newtheorem{df}[thm]{Definition}

\newtheorem{clm}{Claim}[thm]

\theoremstyle{remark}
\newtheorem{rem}[thm]{Remark}

\theoremstyle{plain}

\theoremstyle{definition}

\numberwithin{equation}{section}

\newcommand{\set}[2]{\{#1\,:\,\text{#2}\}} 
\newcommand{\m}[1]{{\mathbf{\uppercase{#1}}}}

\DeclareMathOperator{\sg}{Sg}
\DeclareMathOperator{\Con}{Con}
\DeclareMathOperator{\Pol}{Pol}

\newcommand{\ra}{\rightarrow}
\newcommand{\tup}[1]{\mathbf{#1}}             
\newcommand{\calV}{\mathcal V}                
\newcommand{\proj}{\operatorname{pr}}

\newcommand{\Do}{D^o}

\newcommand{\tr}{\operatorname{tr}}
\newcommand{\reftr}{\operatorname{ref\hspace{-1pt}.\hspace{-2pt}tr}}

\newcommand{\graph}{\operatorname{graph}}
\newcommand{\HSP}{\mathsf{HSP}}
\newcommand{\HS}{\mathsf{HS}}

\newcommand{\Grp}[1]{\operatorname{Grp}_{\m #1}}

\newcommand{\typ}{\operatorname{typ}}

\newcommand{\barsigma}{\overline{\sigma}}
\newcommand{\barrho}{\overline{\rho}}
\newcommand{\baralpha}{\overline{\alpha}}

\newcommand{\barpi}{\overline{\pi}}

\newcommand{\barbeta}{\overline{\beta}}

\newcommand{\barT}{\overline{T}}
\newcommand{\barA}{\overline{A}}
\newcommand{\barB}{\overline{B}}
\newcommand{\barma}{\overline{\m a}}
\newcommand{\barmb}{\overline{\m b}}

\newcommand{\coverdelta}{\delta^+}

\newcommand{\coverrho}{\rho^+}
\newcommand{\coversigma}{\sigma^+}

\newcommand{\coveralpha}{\alpha^+}
\newcommand{\coverzero}{0^+}

\newcommand{\Dmon}{\varphi}     

\newcommand{\montwo}{\kappa} 
\newcommand{\Opt}{\operatorname{Opt}}
\newcommand{\TD}[1]{T_{\m #1}^{D}}

\newcommand{\Topt}{T^{\,\mathsf{opt}}}

\newcommand{\nonempty}{non-empty}
\newcommand{\Implies}{\,$\Rightarrow$\,}

\newcommand{\Iff}{\,$\Leftrightarrow$\,}

\newcommand{\taurho}{\tau}
\newcommand{\genCov}{\operatorname{Cov}}
\newcommand{\Cov}{\operatorname{Cov}^+}

\newcommand{\proper}{similarity}

\newcommand{\rooted}{good}
\newcommand{\Rooted}{Good}

\newcommand{\Conp}{\Con^\ast}

\newcommand{\Ponelmmaltsev}{Lemma 3.1}
\newcommand{\Ponelmgumm}{Lemma 3.5}
\newcommand{\Ponelmpoly}{Lemma 3.10}
\newcommand{\Ponelmfacesitone}{Lemma 4.4(1)}
\newcommand{\PonecorDA}{Corollary 4.8}
\newcommand{\Ponethmpersp}{Theorem 7.7}
\newcommand{\Ponecorsimprop}{Corollary 7.8(2)}

\begin{document}
\bibliographystyle{siam}

\title{Zhuk's bridges, centralizers, and similarity}

\author{Ross Willard}
\address{Pure Mathematics Department, University of Waterloo, Waterloo,
ON N2L 3G1 Canada}
\email{ross.willard@uwaterloo.ca}
\urladdr{www.math.uwaterloo.ca/$\sim$rdwillar}

\thanks{The support of the Natural Sciences and Engineering Research Council (NSERC) of Canada is gratefully acknowledged.}


\date{Version 2 -- January 19, 2026}

\begin{abstract}
This is the second of three papers motivated by the author's desire to 
understand and explain ``algebraically" 
one aspect of Dmitriy Zhuk's proof of the CSP Dichotomy Theorem. 
In this paper
we extend Zhuk's ``bridge" 
construction to arbitrary meet-irreducible congruences 
of finite algebras in locally finite varieties with a Taylor term.
We then connect bridges to centrality and similarity. 
In particular, we prove that Zhuk's
bridges and our ``\proper\ bridges" (defined in our first paper) convey
the same information in locally finite Taylor varieties.
\end{abstract}

\maketitle

\section{Introduction} \label{sec:intro}

Arguably the most important result in universal algebra in the last ten years is the
positive resolution to the Constraint Satisfaction Problem (CSP) Dichotomy Conjecture,
announced independently in 2017 by Andrei Bulatov \cite{bulatov2017} and
Dmitriy Zhuk \cite{zhuk2017,zhuk2020}.
One particular feature of Zhuk's proof is his analysis of 
``rectangular critical" subdirect products $R\leq_{sd} \m a_1\times \cdots \times \m a_n$ of
finite algebras $\m a_1,\ldots,\m a_n$ in certain locally finite idempotent Taylor varieties.  
Zhuk showed that 
such relations $R$ induce derived relations, which he named ``bridges," 
between certain meet-irreducible congruences of $\m a_1,\ldots,\m a_n$
which $R$ determines.
Zhuk also established a number of useful properties of his bridges, and ultimately used them to tease 
out implicit linear equations in CSP instances.  In this and two companion
papers \cite{similar,critical}, we aim to understand ``algebraically"
Zhuk's bridges and their
application to rectangular critical relations.

Our goal in this paper 
is to establish precise connections between Zhuk's bridges,
centrality, and a relation called ``similarity" due to Freese
\cite{freese1983} in the congruence modular setting and extended to varieties
with a weak difference term (including locally finite Taylor varieties)
in our first paper \cite{similar}.
In Section~\ref{sec:defs} we give the basic definitions and
tools needed in this paper.  In Section~\ref{sec:sim} we
summarize the results about similarity from \cite{similar} which we will use
here.
With these preliminaries out of the way, we address two technical
limitations of Zhuk's original presentation:
Zhuk defined his bridges between pairs $(\m a,\rho)$ and $(\m b,\sigma)$
where (i) $\m a$ and $\m b$ are finite algebras in a very special kind of Taylor variety, and
(ii) $\rho$ and $\sigma$ are congruences satisfying an ``irreducibility"
property stronger than meet-irreducibility.
In fact, Zhuk's definitions, and all but one of Zhuk's theorems about bridges
(see Theorem~\ref{thm:zeta}),
work in arbitrary locally finite Taylor varieties, so 
in Section~\ref{sec:bridge} we present his definitions and basic results
avoiding limitation (i).
Then in Section~\ref{sec:taylor} we use tame congruence theory to
show how Zhuk's definitions and basic results can
extend to arbitrary meet-irreducible congruences, avoiding limitation (ii).
Finally, in Section~\ref{sec:central}
we align Zhuk's bridges, in this broader context, 
with the algebraic relations of centrality and similarity.
In particular, we show that between irreducible congruences,
the existence of a Zhuk bridge is equivalent to the existence of
our ``\proper\ bridge"
defined in \cite{similar} (see Definition~\ref{df:bridge}).

\section{Definitions and helpful results} \label{sec:defs}

We assume that the reader is familiar with the fundamentals of universal algebra as given
in \cite{burris-sanka}, \cite{alvin1} or \cite{bergman}.
Our notation generally follows that in \cite{alvin1}, \cite{bergman} and \cite{kearnes-kiss}.
If $\m a$ is an algebra, then $\Con(\m a)$ denotes its congruence lattice.
The smallest and largest congruences of $\m a$ are the
\emph{diagonal} $0_A:=\set{(a,a)}{$a \in A$}$ and the \emph{full congruence}
$1_A:=A^2$
respectively, which will be denoted $0$ and $1$ if no confusion arises.
If $\alpha,\beta \in \Con(\m a)$ with $\alpha\leq\beta$, then $\beta/\alpha$
denotes the congruence of $\m a/\alpha$ corresponding to $\beta$ via the
Correspondence Theorem (\cite[Theorem 4.12]{alvin1} or \cite[Theorem 3.6]{bergman}).
If in addition $\gamma,\delta \in \Con(\m a)$ with $\gamma\leq\delta$, then we write
$(\alpha,\beta)\nearrow (\gamma,\delta)$ 
if $\beta \wedge \gamma = \alpha$ and $\beta \vee \gamma = \delta$.  The notation
$(\gamma,\delta)\searrow (\alpha,\beta)$ means the same thing.
We write $\alpha\prec \beta$ and say that $\beta$ \emph{covers} $\alpha$, 
or is an \emph{upper cover of} $\alpha$, if $\alpha<\beta$ and there does not exist
a congruence $\gamma$ satisfying $\alpha<\gamma<\beta$.
A congruence $\alpha$ is \emph{minimal} if it covers $0$, and is
\emph{completely meet-irreducible} 
if $\alpha\ne 1$ and there exists $\coveralpha$
with $\alpha\prec\coveralpha$ and such that $\alpha<\beta \implies 
\coveralpha \leq \beta$ for all $\beta \in \Con(\m a)$.
When $\m a$ is finite, we use the phrase ``\emph{meet-irreducible}"
to mean the same thing.  We say that $\m a$ is \emph{subdirectly irreducible}
if $0$ is completely meet-irreducible, in which case $\coverzero$ is 
called the \emph{monolith} of $\m a$.
A subset $T\subseteq A$ is a \emph{transversal} for a congruence $\alpha$ if 
it contains exactly one element from each $\alpha$-class.

If $f:A\ra B$ is a function, then its \emph{graph} is the set
$\graph(f)=\set{(a,f(a))}{$a \in A$}$.
If $n>0$, then $[n]$ denotes $\{1,2,\ldots,n\}$.
If $R\subseteq A_1\times \cdots \times A_n$ and $i,j \in [n]$, then
$\proj_i(R)$ denotes the projection of $R$ onto its $i$-th coordinate,
and $\proj_{i,j}(R)$ denotes $\set{(a_i,a_j)}{$(a_1,\ldots,a_n) \in R$}$.

We follow \cite{kearnes-kiss} and refer to a set of operation symbols with assigned
arities as a \emph{signature}.  
Every algebra comes equipped with a signature, which indexes the 
\emph{basic operations} of the algebra.
\emph{Terms} are formal recipes for constructing new operations from the
basic operations
via composition and variable manipulations; see
\cite{alvin1} or \cite{bergman} or any textbook on first-order logic.
We will not distinguish between terms and the term operations they
define in an algebra, except when the distinction is crucial.  
A \emph{polynomial} of an algebra $\m a$ is any operation on $A$ having the form
$t(x_1,\ldots,x_n,\tup c)$ where $t$ is an $(n+k)$-ary term 
in the signature of $\m a$ and $\tup c \in A^k$.  $\Pol_n(\m a)$ is the set of 
all $n$-ary polynomials of $\m a$.

If $\m a$ is an algebra, a term $t(x_1,\ldots,x_n)$ of $\m a$ is
\emph{idempotent} if $\m a$ satisfies the identity $t(x,\ldots,x)\approx x$, 
and is a \emph{Taylor term}
if it is idempotent, $n>1$, and for each $i \in [n]$, $\m a$ satisfies an identity
of the form 
\[
t(u_1,\ldots,u_n) \approx t(v_1,\ldots,v_n)
\]
where each $u_j$ and $v_k$ is the variable $x$ or $y$, and 
$\{u_i,v_i\}=\{x,y\}$.
An algebra is said to be \emph{Taylor} if it has a Taylor term.
A particularly important example of a Taylor term is a \emph{weak near-unanimity term}
(WNU),
which is an $n$-ary idempotent term $w(x_1,\ldots,x_n)$ with $n>1$ which 
satisfies the identities
\[
w(y,x,x,\ldots,x) \approx w(x,y,x,\ldots,x) \approx w(x,x,y,\ldots,x) \approx \cdots \approx
w(x,\ldots,x,y).
\]
Another important example of a Taylor term is a \emph{Maltsev term}; this is a ternary
term $p(x,y,z)$ satisfying the identities
\begin{equation} \label{eq:maltsev}
p(x,x,y) \approx y \approx p(y,x,x).
\end{equation}
The identities \eqref{eq:maltsev} are called the \emph{Maltsev identities}.  Any 
ternary operation
(whether a term or not) satisfying them is called a \emph{Maltsev operation}.

A \emph{variety} is a class of algebras (in a common signature) which is closed under 
subalgebras, homomorphic images, and direct products of arbitrary (including
infinite) families of algebras.  $\HSP(\m a)$ denotes the smallest variety containing $\m a$.
A term is a Taylor term, or a WNU term, for a variety if it is such for
every algebra in the variety.  Because the definitions of Taylor terms and WNU terms 
are given in terms of satisfied identities, a Taylor or WNU term for an algebra $\m a$ is 
automatically a Taylor or WNU term for the variety $\HSP(\m a)$.

Before defining ``weak difference term," we recall the ternary centralizer relation on
congruences and the notion of
abelian congruences.
Given a \nonempty\ set $A$, let $A^{2\times 2}$ denote the set of all
$2\times 2$ matrics over $A$.  If $\m a$ is an algebra, let 
$\m a^{2\times 2}$ denote the algebra with universe $A^{2\times 2}$ which is 
isomorphic to $\m a^4$ via the bijection
\[
\begin{pmatrix} a_1&a_3\\a_2&a_4\end{pmatrix} \mapsto (a_1,a_2,a_3,a_4).
\]

\begin{df} \label{df:M}
Suppose $\m a$ is an algebra and $\theta,\varphi \in \Con(\m a)$.
$M(\theta,\varphi)$ is the subuniverse of $\m a^{2\times 2}$ generated by the
set
\[
X(\theta,\varphi):=\left\{ \begin{pmatrix} c&c\\d&d\end{pmatrix} : (c,d) \in \theta\right\}
\cup
\left\{ \begin{pmatrix} a&b\\a&b\end{pmatrix} : (a,b) \in \varphi\right\}.
\]
The matrices in $M(\theta,\varphi)$ are called $(\theta,\varphi)$-\emph{matrices}.
\end{df}

\begin{df} \label{df:cent}
Suppose $\theta,\varphi,\delta \in \Con(\m a)$.  We say that 
$\varphi$ \emph{centralizes $\theta$ modulo $\delta$}, and write
$C(\varphi,\theta;\delta)$,
if any of the following equivalent conditions holds:
\begin{enumerate}
\item\label{dfcent:it1}
For every matrix in $M(\varphi,\theta)$, if one row is in $\delta$, then so is the other row.
\item \label{dfcent:it2}
For every matrix in $M(\theta,\varphi)$, if one column is in $\delta$, then so is the 
other column.
\item \label{dfcent:it3}
For every $(1+n)$-ary term $t(x,y_1,\ldots,y_n)$ and all $(a,b) \in \varphi$
and $(c_j,d_j) \in \theta$ ($j \in [n]$),
\[
\mbox{if 
$t(a,\tup{c})\stackrel{\delta}{\equiv} t(a,\tup{d})$, then
$t(b,\tup{c})\stackrel{\delta}{\equiv} t(b,\tup{d})$.}
\]
\end{enumerate}
\end{df}

Let $\theta,\delta$ be congruences of an algebra $\m a$.  
The \emph{centralizer} (or
\emph{annihilator}) \emph{of $\theta$ modulo $\delta$}, denoted
$(\delta:\theta)$, is the unique largest congruence
$\varphi$ for which $C(\varphi,\theta;\delta)$ holds.
We say that $\theta$ is \emph{abelian} if $C(\theta,\theta;0)$ holds; equivalently,
if $\theta \leq (0:\theta)$.  
We say that $\m a$ is \emph{abelian} if $1_A$ is abelian.
More generally, if $\theta,\delta \in 
\Con(\m a)$ with
$\delta\leq\theta$, then we say that $\theta$ is \emph{abelian modulo} $\delta$ if
any of the following equivalent conditions holds: (i) $\theta/\delta$ is an abelian
congruence of $\m a/\delta$; (ii) $C(\theta,\theta;\delta)$ holds; (iii) 
$\theta \leq
(\delta:\theta)$.  

\begin{df}
Let $\calV$ be a variety, $\m a\in \calV$, and $d(x,y,z)$ a ternary term in the
signature of $\calV$.
\begin{enumerate}
\item
$d$ is a \emph{weak difference term} for $\m a$ if $d$ is idempotent and for
every pair $\delta,\theta$ of congruences with $\delta\leq \theta$ and $\theta/\delta$
abelian, we have
\begin{equation} \label{eq:wdt}
\mbox{$d(a,a,b) \stackrel{\delta}{\equiv} b \stackrel{\delta}{\equiv} d(b,a,a)$
for all $(a,b) \in \theta$.}
\end{equation}
\item
$d$ is a \emph{weak difference term} for $\calV$ if it is a weak difference term for
every algebra in $\calV$.
\end{enumerate}
\end{df}

Note in particular that if $d$ is a weak difference term for $\m a$ and $\theta$ is
an abelian congruence,
then setting $\delta=0$ in \eqref{eq:wdt} gives that
the restriction of $d$ to any $\theta$-class is a Maltsev operation on that class.
In fact, $d$ induces an abelian group operation on each $\theta$-class in this case,
by the following result of Gumm \cite{gumm-maltsev} and Herrmann \cite{herrmann1979}.

\begin{df}
Suppose $\m a$ is an algebra having a weak difference term $d(x,y,z)$,
and $\theta$ is an abelian congruence of $\m a$.  
Given $e \in A$, let $\Grp a(\theta,e)$ denote the algebra $(e/\theta,+,e)$ whose
universe is the $\theta$-class containing $e$ and whose two operations are 
the binary operation $x+y:= d(x,e,y)$ and the constant $e$.
\end{df}

\begin{lm} [essentially \cite{gumm-maltsev,herrmann1979}; cf.\ \mbox{\cite[\Ponelmgumm]{similar}}] \label{lm:gumm}
Suppose $\m a,d,\theta$ are as in the previous definition and
$e \in A$.
Then $\Grp a(\theta,e)$ is an abelian group with zero element $e$.
Moreover, we have $-x=d(e,x,e)$ and $d(x,y,z)=x-y+z$ for all $x,y,z \in e/\theta$.
\end{lm}

The next result is folklore.

\begin{lm} [cf.\ \mbox{\cite[\Ponelmmaltsev]{similar}}] \label{lm:maltsev}
Suppose $\m a$ is an algebra and $\rho$ is a reflexive subuniverse of $\m a^2$.
Suppose $\m a$ has a ternary 
term $d(x,y,z)$ such that for all $(a,b) \in \rho$ we have
$d(a,a,b)=b$ and $d(a,b,b)=a$.
Then $\rho\in\Con(\m a)$.
\end{lm}

The next lemma will be needed in Section~\ref{sec:central}.

\begin{lm} \label{lm:poly} \cite[\Ponelmpoly]{similar}
Suppose $\m a$ is an algebra having a weak difference term,
$\mu \in \Con(\m a)$ is an abelian minimal congruence, and $S\leq \m a$ is a
subuniverse of $\m a$ which is a transversal for $\mu$.  
Then $S$ is a maximal proper subuniverse of $\m a$.
\end{lm}

Some proofs in Sections~\ref{sec:taylor} and~\ref{sec:central}
use tame congruence theory, of which we will
give only a cursory overview.  
Given a finite algebra $\m a$, to each pair $(\alpha,\beta)$ of congruences 
of $\m a$ with $\alpha\prec
\beta$, the theory assigns one of five ``types" from
$\{\tup 1,\tup 2,\tup 3,\tup 4,\tup 5\}$.
The theory also defines ``$(\alpha,\beta)$-minimal sets," which are special 
subsets of the universe $A$, and ``$(\alpha,\beta)$-traces," which are sets of
the form $U\cap C$ where $U$ is an $(\alpha,\beta)$-minimal set, $C$ is a $\beta$-class,
and $(U\cap C)^2\nsubseteq \alpha$.
Most of what we will need is contained in the following two results.

\begin{prp}[\cite{tct}] \label{prp:tct}
Suppose $\m a$ is a finite algebra and $\alpha,\beta \in \Con(\m a)$ with
$\alpha \prec \beta$.
\begin{enumerate}
\item \label{tct:it1}
There exists an $(\alpha,\beta)$-minimal set.
\item \label{tct:it2}
$\beta/\alpha$ is nonabelian if and only if $\typ(\alpha,\beta)\in \{\tup 3,\tup 4,\tup 5\}$.
\item \label{tct:it3}
Suppose $\beta/\alpha$ is nonabelian and $U$ is an $(\alpha,\beta)$-minimal set.
Then there exists a unique
$\beta$-class $C$ such that $U\cap C$ is an $(\alpha,\beta)$-trace.
Moreover, letting $N:=U\cap C$, there exist $1 \in N$,
a unary polynomial $e(x) \in \Pol_1(\m b)$, and a binary polynomial $p(x,y) \in
\Pol_2(\m a)$ satisfying:
\begin{enumerate}
\item
$(N\setminus \{1\})^2 \subseteq \alpha$.
\item
$e(A)=U$ and $e(x)=x$ for all $x \in U$.
\item
$p(x,1)=p(1,x)=p(x,x)=x$ for all $x \in U$.
\item
$p(x,o) \stackrel{\alpha}{\equiv} p(o,x) \stackrel\alpha\equiv x$ for all
$x \in U\setminus \{1\}$ and all $o \in N\setminus \{1\}$.
\end{enumerate}
\end{enumerate}
\end{prp}

\begin{proof}
\eqref{tct:it1} follows from \cite[Theorem 5.7(1) and Theorem 2.8(2)]{tct}.
\eqref{tct:it2} follows from \cite[Theorem 5.7(1,2)]{tct}.
\eqref{tct:it2} can be deduced from \cite[Theorem 5.7(1) and Lemmas 2.13(3), 
4.15 and 4.17]{tct}.
\end{proof}

A variety is \emph{locally finite} if its finitely generated algebras are
all finite.  
In particular, $\HSP(\m a)$ is locally finite whenever
$\m a$ is finite.  
A locally finite variety 
\emph{omits type $\tup i$} if no finite algebra in the variety has a
pair of congruences $\alpha\prec \beta$ with $\typ(\alpha,\beta)= \tup i$.

\begin{thm} \label{thm:locfin}
For a locally finite variety $\calV$, the following are equivalent:
\begin{enumerate}
\item \label{locfin:it1}
$\calV$ has a Taylor term.
\item \label{locfin:it2}
$\calV$ has a WNU term.
\item \label{locfin:it3}
$\calV$ has a weak difference term.
\item \label{locfin:it4}
$\calV$ omits type $\tup 1$.
\end{enumerate}
\end{thm}

\begin{proof}
This follows by combining \cite[Theorem 9.6]{tct},
\cite[Corollary 5.3]{taylor}, \cite[Theorem 2.2]{maroti-mckenzie},
and \cite[Theorem 4.8]{kearnes-szendrei}.
\end{proof}

\section{Similarity in varieties with a weak difference term}
\label{sec:sim}

In this section we list the definitions and results about similarity from 
\cite{similar} which we will use in Section~\ref{sec:central}.

\begin{df} \label{df:A(mu)}
Suppose $\m a$ is an algebra and $\theta,\alpha \in \Con(\m a)$
with $\theta\leq\alpha$.
\begin{enumerate}
\item \label{A(mu):it1}
$\m a(\theta)$ denotes $\theta$ viewed as a subalgebra of $\m a^2$.
\item \label{A(mu):it3}
$\Delta_{\theta,\alpha}$ denotes
the congruence of $\m a(\theta)$ generated by $\set{((a,a),(b,b))}{$(a,b) \in 
\alpha$}$.
\end{enumerate}
\end{df}

\begin{lm} \label{lm:Deltaprop}
Let $\m a$ be an algebra and $\delta,\theta,\alpha \in \Con(\m a)$
with $\theta\leq\alpha$ and $C(\alpha,\theta;\delta)$.  Then for all
$((a,a'),(b,b')) \in \Delta_{\theta,\alpha}$,
$(a,a') \in \delta \iff (b,b')\in \delta$.
\end{lm}

\begin{proof}
$\Delta_{\theta,\alpha}$ is the transitive closure of
$\set{((a,a'),(b,b'))}{${\scriptsize{\begin{pmatrix}a\phantom{'}&b\phantom{'}\\a'&b'\end{pmatrix}}} \in M(\theta,
\alpha)$}$ (see e.g.\ \cite[\Ponelmfacesitone]{similar}).  Hence the result
follows by $C(\alpha,\theta;\delta)$ and 
Definition~\ref{df:cent}\eqref{dfcent:it2}.
\end{proof}

We are mainly interested in $\Delta_{\theta,\alpha}$ when 
$\theta$ is abelian and $\alpha=(0:\theta)$.  In this context we use the
following notation.

\begin{df} \label{df:D(A,theta)}
Suppose $\m a$ is an algebra,
$\theta \in \Con(\m a)$ is abelian, and $\alpha=(0:\theta)$.
\begin{enumerate}
\item
$\baralpha$ denotes the set 
$\set{((a,a'),(b,b')) \in A(\theta)^2}{$(a,b) \in \alpha$}$, which
is a congruence of $\m a(\theta)$ satisfying $\Delta_{\theta,\alpha}\leq \baralpha$.
\item
$D(\m a,\theta)$ denotes the quotient algebra $\m a(\theta)/\Delta_{\theta,\alpha}$.
\end{enumerate}
\end{df}

\begin{thm} [\mbox{\cite[\PonecorDA]{similar}}] \label{cor:D(A)}
Suppose $\m a$ belongs to a variety with a weak difference term and
$\theta$ is an abelian minimal congruence of $\m a$.  Let
$\alpha=(0:\theta)$ and $\Dmon = \baralpha/\Delta_{\theta,\alpha}$.
Then $D(\m a,\theta)$ is subdirectly irreducible with abelian monolith $\Dmon$.
Moreover, $(0:\Dmon)=\Dmon$, 
and there exist a surjective homomorphism
$h:\m a(\theta) \ra D(\m a,\theta)$, 
an isomorphism $h:\m a/\alpha \cong D(\m a,\theta)/\Dmon$,
and a subuniverse $\Do\leq D(\m a,\theta)$ such that:
\begin{enumerate}
\item \label{corD(A):it1}
$\Do$ is a transversal for $\Dmon$.
\item \label{corD(A):it2}
$h^{-1}(\Do) = 0_A$.
\item \label{corD(A):it3}
For all $(a,b) \in \theta$,
$h(a,b)/\Dmon = h^\ast(a/\alpha)$.
\end{enumerate}
\end{thm}

\begin{df} \label{df:D(A)}
Suppose $\calV$ is a variety with a weak difference term, and
$\m a \in \calV$ is subdirectly irreducible with monolith $\mu$.  The algebra $D(\m a)$ is
defined as follows:
\begin{enumerate}
\item
If $\mu$ is nonabelian, then $D(\m a)=\m a$.
\item
If $\mu$ is abelian, then $D(\m a)=D(\m a,\mu)$ as defined in Definition~\ref{df:D(A,theta)}.
\end{enumerate}
\end{df}

\begin{df} \label{df:similar}
Suppose $\calV$ is a variety with a weak difference term, and
$\m a,\m b \in \calV$ are subdirectly irreducible.  We say that $\m a$ and $\m b$ are
\emph{similar}, and write $\m a\sim \m b$, if $D(\m a)\cong D(\m b)$.
\end{df}

The following definition from \cite{similar} was motivated by Zhuk's bridges
\cite{zhuk2020}.

\begin{df} \label{df:bridge}
Suppose $\m a,\m b$
are subdirectly irreducible algebras in a common signature with monoliths
$\mu,\montwo$ respectively.  A \emph{\proper\ bridge} from $\m a$ to $\m b$
is a subuniverse $T\leq\m a\times \m a\times \m b\times \m b$ satisfying
\begin{enumerate}[label=(B\arabic*)]
\item \label{bridge:it1}
$\proj_{1,2}(T)=\mu$ and $\proj_{3,4}(T)=\montwo$.
\item \label{bridge:it2}
For all $(a_1,a_2,b_1,b_2) \in T$ we have $a_1=a_2$ if and only if $b_1=b_2$.
\item \label{bridge:it3}
For all $(a_1,a_2,b_1,b_2)\in T$ we have $(a_i,a_i,b_i,b_i) \in T$ for $i=1,2$.
\end{enumerate}
\end{df}

\begin{thm} [\mbox{\cite[\Ponethmpersp]{similar}}] \label{thm:persp}
Suppose $\calV$ is a variety with a weak difference term, and
$\m a,\m b\in \calV$ are subdirectly irreducible.
The following are equivalent:
\begin{enumerate}
\item \label{persp:it1}
$\m a \sim \m b$.
\item \label{persp:it2}
There exist an algebra $\m c \in \calV$,
surjective homomorphisms $f_1:\m c\ra \m a$ and $f_2:\m c\ra \m b$,
 and congruences 
$\psi,\tau \in \Con(\m C)$ with $\psi<\tau$, such that, letting $\delta_i = \ker(f_i)$ and 
letting $\coverdelta_i$
denote the unique upper cover of $\delta_i$ in $\Con(\m c)$ for $i=1,2$, we have
$(\delta_1,\coverdelta_1) \searrow (\psi,\tau) \nearrow (\delta_2,\coverdelta_2)$.

\item \label{persp:it4}
There exists a \proper\ bridge from $\m a$ to $\m b$.
\end{enumerate}
\end{thm}

\begin{cor} [\mbox{\cite[\Ponecorsimprop]{similar}}] \label{cor:simprop}
Suppose $\calV$ is a variety with a weak difference term, and
$\m a\in \calV$ is subdirectly irreducible with abelian monolith $\mu$.
Setting $\alpha = (0:\mu)$ and $\Delta = \Delta_{\mu,\alpha}$, the set 
\[
\TD a := \set{(a,b,(a,e)/\Delta,(b,e)/\Delta)}{$a\stackrel{\mu}{\equiv}
b \stackrel{\mu}{\equiv}e$}
\]
is a \proper\ bridge from $\m a$ to $D(\m a)$.
\end{cor}

\section{Zhuk's bridges} \label{sec:bridge}

In his solution to the 
Constraint Satisfaction Problem Dichotomy Conjecture \cite{zhuk2020}, D. Zhuk
defined and used to great effect certain relations which he called ``bridges."
In this section we present Zhuk's bridges, in slightly greater generality
than Zhuk's original setting.

\begin{df} \label{df:saturated}
Let $\m a$ be an algebra and $\rho \in \Con(\m a)$.
\begin{enumerate}
\item
A subuniverse $R\leq \m a^2$ is said to be \emph{stable under}
$\rho$ \cite{zhuk2020}, or $\rho$-\emph{saturated} \cite{kearnes-szendrei-parallel}, 
or $\rho$-\emph{closed} \cite{tct},
if $R=\rho\circ R\circ \rho$.
\item
More generally, if $\m a_1,\ldots,\m a_n$ are algebras in the same signature
as $\m a$, $R\leq \m a_1\times \cdots \times \m a_n$, $i \in [n]$, and $\m a_i=\m a$, then $R$ is \emph{stable under $\rho$ in coordinate $i$} if
$(a_1,\ldots,a_i,\ldots,a_n) \in R$ and $(a_i,b_i) \in \rho$ imply
$(a_1,\ldots,b_i,\ldots,a_n) \in R$.
\end{enumerate}
\end{df}

\begin{df}
If $\m a$ is an algebra, then $\Conp(\m a)$ denotes
$\Con(\m a)\setminus \{1_A\}$.
\end{df}

\begin{df} \label{df:pbridge}
Let $\m a,\m b$ be finite algebras in a common signature,
$\rho \in \Conp(\m a)$, and $\sigma \in \Conp(\m b)$.
A \emph{bridge from $(\m a,\rho)$ to $(\m b,\sigma)$} is
a subuniverse $T\leq \m a\times \m a \times \m b \times \m b$
satisfying
\begin{enumerate}[label=(B\arabic*${}^\ast$)]
\setcounter{enumi}{-1}
\item \label{pbridge:it0}
$T$ is stable under $\rho$ in its first two coordinates and stable
under $\sigma$ in its last two coordinates.
\item \label{pbridge:it1}
$\rho \subset \proj_{1,2}(T)$ and $\sigma \subset \proj_{3,4}(T)$.
\item \label{pbridge:it2}
For all $(a_1,a_2,b_1,b_2) \in T$ we have 
$(a_1,a_2)\in \rho\iff (b_1,b_2) \in \sigma$.
\end{enumerate}
If $T$ is a bridge from $(\m a,\rho)$ to $(\m b,\sigma)$, then we call
$L:=\proj_{1,2}(T)$ and $R:=\proj_{3,4}(T)$ the \emph{left} and \emph{right anchors}
of $T$, and say that $T$ is a bridge \emph{from $(\m a,\rho,L)$ to
$(\m b,\sigma,R)$.}
\end{df}

Note that if $T$ is a bridge from $(\m a,\rho)$ to $(\m b,\sigma)$, then the
left and right anchors of $T$ are stable under $\rho$ and $\sigma$ respectively.
Also note that
condition~\ref{pbridge:it0} allows for 
``faithfully modding out by $\rho$ and $\sigma$."  That is, if
$T,\rho,\sigma$ satisfy
\ref{pbridge:it0}, then setting $\overline{\m a}:= \m a/\rho$, 
$\overline{\m b}:= \m b/\sigma$, and
$\overline{T} = \set{(a/\rho,\,a'/\rho,\,b/\sigma,\,b'/\sigma)}{$
(a,a',b,b') \in T$}$,
we get that
$T$ is a bridge from  $(\m a,\rho)$ to $(\m b,\sigma)$ if and only
if $\overline{T}$ is a bridge from $(\overline{\m a},0_{\overline{A}})$
to $(\overline{\m b},0_{\overline{B}})$.
Moreover, $T$ is recoverable from $\overline{T}$, namely, as
the pre-image under the natural map
$\m a^2\times \m b^2 \ra (\overline{\m a})^2\times(\overline{\m b})^2$.

\begin{df} \label{df:bridgeops}
Suppose $\m a,\m b,\m c$ are finite algebras in a common signature
and $\rho \in \Conp(\m a)$, $\sigma \in \Conp(\m b)$, and
$\upsilon \in \Conp(\m c)$.
Let $\rho \subset L\leq \m a^2$ with $L$ stable under $\rho$,
let $T$ be a bridge from $(\m a,\rho)$ to $(\m b,\sigma)$,
and let $T'$ be a bridge from $(\m b,\sigma)$ to $(\m c,\upsilon)$.
\begin{enumerate}
\item
The \emph{identity bridge} for $(\m a,\rho,L)$ is the
relation 
\[
I_{(\m a,\rho,L)} = \set{(a_1,a_2,b_1,b_2) \in A^4}{$(a_1,a_2),(b_1,b_2)\in L$ and
$(a_1,b_1),(a_2,b_2) \in \rho$}.
\]
\item
The \emph{converse} of $T$ is the subuniverse $T^\cup \leq
\m b\times \m b\times \m a\times \m a$ given by
\[
T^\cup = \set{(b_1,b_2,a_1,a_2)}{$(a_1,a_2,b_1,b_2) \in T$}.
\]
\item \label{bridgeops:it3}
The \emph{composition} $T\circ T'$ is the subuniverse
$T\circ T'\leq\m a\times \m a\times \m c\times \m c$ given by
\[
T\circ T' = \set{(a_1,a_2,c_1,c_2)}{$\exists b_1,b_2\in B$ with
$(a_1,a_2,b_1,b_2)\in T$ and $(b_1,b_2,c_1,c_2) \in T'$}.
\]
\end{enumerate}
\end{df}

It is easy to check that in the context of Definition~\ref{df:bridgeops},
the identity bridge $I_{(\m a,\rho,L)}$ is a bridge from $(\m a,\rho)$
to itself with both anchors equal to $L$, and
the converse $T^\cup$ is a bridge from $(\m b,\sigma)$ to $(\m a,\rho)$
whose left and right anchors are the right and left anchors, respectively,
of $T$.
The composition $T\circ T'$ however need not be a bridge from
$(\m a,\rho)$ to $(\m c,\upsilon)$, as it will satisfy~\ref{pbridge:it1}
only when the intersection of the right anchor of $T$ with the left
anchor of $T'$ properly contains $\sigma$.

An important invariant of a bridge is its ``trace."

\begin{df} \label{df:trace}
Suppose $T$ is a bridge from $(\m a,\rho)$ to $(\m b,\sigma)$.
\begin{enumerate}
\item
The \emph{trace} of $T$, denoted $\tr(T)$,
is the subuniverse of $\m a\times \m b$ defined by
$\tr(T)=\set{(a,b)}{$(a,a,b,b)\in T$}$.
\item
When $\m b=\m a$, we say that $T$ is \emph{reflexive} if
$0_A\subseteq \tr(T)$.
\end{enumerate}
\end{df}

Zhuk \cite{zhuk2020} denoted $\tr(T)$ by $\widetilde{T}$ (and did not
call it a ``trace").
One can check that
$\tr(I_{(\m a,\rho,L)})=\rho$, $\tr(T^\cup) =  \tr(T)^{-1}$,
and $\tr(T\circ T')=\tr(T)\circ \tr(T')$.

The next two definitions (not from \cite{zhuk2020}) will help us articulate
a key fact about bridges: they can be restricted to ``minimal" anchors without
affecting their trace.

\begin{df}
Let $\m a$ be a finite algebra and $\rho \in \Conp(\m a)$.
$\genCov(\rho)$ denotes the set of minimal (under inclusion) 
$\rho$-saturated subuniverses of $\m a^2$ properly containing $\rho$.
\end{df}

\begin{df}
A bridge from $(\m a,\rho)$ to $(\m b,\sigma)$ is \emph{compact}
if its left anchor is in $\genCov(\rho)$ and its right anchor is in
$\genCov(\sigma)$.
\end{df}

\begin{lm} 
\label{lm:goodbridge}
Suppose $\m a,\m b$ are finite algebras in a common signature,
$\rho \in \Conp(\m a)$, $\sigma \in \Conp(\m b)$, and $T$ is a bridge
from $(\m a,\rho)$ to $(\m b,\sigma)$.
For every $L' \in \genCov(\rho)$ with $L'\subseteq \proj_{1,2}(T)$ there exists
a compact bridge $T'\subseteq T$ from $(\m a,\rho)$ to $(\m b,\sigma)$ with
$\proj_{1,2}(T')=L'$ and $\tr(T')=\tr(T)$.
\end{lm}

\begin{proof} 
Define $T_1 = \set{(a_1,a_2,b_1,b_2)\in T}{$(a_1,a_2)\in L'$}$.
As $L'\subseteq \proj_{1,2}(T)$, we get
$L'=\proj_{1,2}(T_1)$.  In particular, there exists 
$(a_1,a_2,b_1,b_2) \in T_1$ with $(a_1,a_2)\not\in \rho$.
By property \ref{pbridge:it2}, this implies $(b_1,b_2)\not\in \sigma$.
It can be checked that
$T_1$ is a bridge from $(\m a,\rho)$ to
$(\m b,\sigma)$ with $\tr(T_1)=\tr(T)$ and $\proj_{1,2}(T_1)=L'$, and clearly
$T_1\subseteq T$.

Because $\proj_{3,4}(T_1)$ is $\sigma$-saturated and properly contains $\sigma$,
we can pick $R' \in \genCov(\sigma)$ with $R' \subseteq \proj_{34}(T_1)$.
Let $T'= \set{(a_1,a_2,b_1,b_2) \in T_1}{$(b_1,b_2) \in R'$}$.
An argument like the one in the previous paragraph shows 
that $T'$ is a bridge from $(\m a,\rho)$ to $(\m b,\sigma)$
with $\tr(T')=\tr(T_1)$, $\proj_{3,4}(T')=R'$, and $T'\subseteq T_1$.
Since $\rho \subset \proj_{1,2}(T')\subseteq \proj_{1,2}(T_1)=L'$, 
$\proj_{1,2}(T')$ is $\rho$-saturated, and
$L' \in \genCov(\rho)$, we get $\proj_{1,2}(T')=L'$
and $T'$ is compact.
\end{proof}

\begin{df} \label{df:ref.tr}
Given a finite algebra $\m a$, $\rho \in \Conp(\m a)$,
and $L \in \genCov(\rho)$, we let $\reftr(\rho,L)$ denote the set 
\[
\reftr(\rho,L)=\set{\tr(T)}{$T$ is a reflexive bridge from $(\m a,\rho,L)$ to
$(\m a,\rho,L)$}.
\]
\end{df}

Note that each member of $\reftr(\rho,L)$ is a subuniverse of $\m a^2$ containing $\rho$,
by \ref{pbridge:it0}.  
In addition, $\reftr(\rho,L)$ is nonempty and closed under inversion and composition by the comments 
following Definitions~\ref{df:bridgeops} and~\ref{df:trace}.
It follows from this and finiteness that $\reftr(\rho,R)$ contains a unique maximal member;
and this unique maximal member is a congruence containing $\rho$.
The following definition and lemma record this observation.

\begin{df}\label{df:opt}
If $\m a$ is a finite algebra, $\rho \in \Conp(\m a)$, and 
$L \in \genCov(\rho)$, then
 $\Opt(\rho,L)$ denotes the unique maximal member of $\reftr(\rho,L)$.
\end{df}

\begin{lm}  \label{lm:opt}
For $\m a,\rho,L$ as in Definition~\ref{df:opt},
we have $\rho\leq \Opt(\rho,L) \in \Con(\m a)$.
\end{lm}

In his proof of the CSP Dichotomy Theorem \cite{zhuk2020},
Zhuk only needed to consider bridges between congruences
$\rho,\sigma$ for which $|\genCov(\rho)|=|\genCov(\sigma)|=1$.

\begin{df} [Zhuk \cite{zhuk2020}] \label{df:irred}
Let $\m a$ be a finite algebra and $\rho \in \Conp(\m a)$.
\begin{enumerate}
\item
$\rho$ is \emph{irreducible} if $|\genCov(\rho)|=1$.
\item
If $\rho$ is irreducible, then $\rho^\ast$ denotes the unique 
member of $\genCov(\rho)$.
\end{enumerate}
\end{df}

The following fact is easily proved.

\begin{lm} \label{lm:meetirred}
If $\m a$ is finite and $\rho \in \Con(\m a)$ is irreducible, then $\rho$ is meet-irreducible
in $\Con(\m a)$; its unique upper cover $\coverrho$ in $\Con(\m a)$ is the 
transitive closure of $\rho^\ast$.
\end{lm}

\begin{df} \label{df:optimal}
[Zhuk \mbox{\cite{zhuk2020}}]
Suppose $\m a$ is a finite algebra and $\rho\in \Con(\m a)$ is irreducible.
\begin{enumerate}
\item
 $\Opt(\rho)$ denotes $\Opt(\rho,\rho^\ast)$.
\item
A bridge $T$ from $(\m a,\rho)$ to $(\m a,\rho)$ is 
\emph{optimal} if $\tr(T)=\Opt(\rho)$.
\end{enumerate}
\end{df}

Zhuk defined $\Opt(\rho)$ in a slightly different way: he let $\Opt(\rho)$ be the
unique maximal member of the set 
\[
\set{\tr(T)}{$T$ is a reflexive bridge from $(\m a,\rho)$ to
$(\m a,\rho)$}.
\]
In fact, this set is identical to $\reftr(\rho,\rho^\ast)$ by Lemma~\ref{lm:goodbridge}, so 
Zhuk's and our definitions of $\Opt(\rho)$ are equivalent.

The following fact can be extracted from Zhuk's proof of \cite[Corollary 7.24.1]{zhuk2020}.

\begin{lm}  \label{lm:modiso}
Suppose $T$ is a compact bridge from $(\m a,\rho)$ to $(\m b,\sigma)$.
Let $L=\proj_{1,2}(T)$, $R=\proj_{3,4}(T)$, 
$\alpha = \Opt(\rho,L)$, and $\beta = \Opt(\sigma,R)$.
Then $\tr(T)$ induces an isomorphism $\gamma:{\m a/\alpha} \cong
\m b/\beta$ defined by
\[
\gamma(a/\alpha)=b/\beta \iff 
(a,b) \in \alpha \circ \tr(T) \circ \beta.
\]
\end{lm}

\begin{proof}
Suppose $(a,b),(a',b') \in \tr(T)$.  It suffices to show
$(a,a')\in \alpha\iff(b,b')\in \beta$.
Assume $(b,b')\in \beta$.  
Let $T_0$ be a reflexive bridge from $(\m b,\sigma,R)$ to itself satisfying
$\tr(T_0)=\beta$.
Let $T_1 = T \circ T_0$ and $T_2 = T_1\circ T_1^\cup$.
Then $(a,b'),(a',b') \in \tr(T)\circ\beta = \tr(T_1)$ and thus $(a,a') \in 
\tr(T_1)\circ \tr(T_1)^{-1}=\tr(T_2)$.
Since $T_2$ is a reflexive bridge from $(\m a,\rho,L)$
to itself, we get $\tr(T_2)\subseteq \Opt(\rho,L)=\alpha$,
so $(a,a') \in \alpha$.  
Thus we have proved $(b,b')\in \beta$ implies $(a,a') \in \alpha$.
A similar proof shows the opposite implication.
\end{proof}

\begin{lm} \label{lm:sameopt-gen}
Suppose $\m a$ is a finite algebra, $\rho,\sigma \in \Conp(\m a)$,
$L \in \genCov(\rho)$, and $R \in \genCov(\sigma)$.
If there exists a reflexive bridge from $(\m a,\rho,L)$ to $(\m a,\sigma,R)$,
then $\Opt(\rho,L)=\Opt(\sigma,R)$.
\end{lm}

\begin{proof}
Let $T_1$ be a reflexive bridge from $(\m a,\rho,L)$ to $(\m a,\sigma,R)$.  
Let $T_2$ be a reflexive bridge from $(\m a,\rho,L)$ to itself satisfying $\tr(T_2)=\Opt(\rho,L)$.  
Let $T_3=T_1^\cup \circ T_2\circ T_1$.  We have
\begin{align*}
\Opt(\rho,L) &= 0_A \circ \Opt(\rho,L)\circ 0_A\\
&\subseteq \tr(T_1^\cup) \circ \tr(T_2) \circ \tr(T_1)\\
&= \tr(T_3)\\
&\subseteq \Opt(\sigma,R)
\end{align*}
where the last inclusion is because $T_3$ is a reflexive bridge from $(\m a,\sigma,R)$
to itself.  A symmetric argument shows $\Opt(\sigma,R)\subseteq \Opt(\rho,L)$.
\end{proof}

\begin{df}[cf.\ Zhuk \cite{zhuk2020}] \label{df:adjacent}
Suppose $\m a$ is a finite algebra and $\rho,\sigma \in \Con(\m a)$ are irreducible.
We say that $\rho$ and $\sigma$ are \emph{adjacent} if there exists a reflexive
bridge from $(\m a,\rho)$ to $(\m a,\sigma)$.
\end{df}

\begin{cor}[Zhuk {\cite[Lemma 6.4]{zhuk2020}}] \label{lm:sameopt}
Suppose $\m a$ is a finite algebra and $\rho,\sigma \in \Con(\m a)$ are
irreducible and adjacent.  Then $\Opt(\rho)=\Opt(\sigma)$.
\end{cor}

\begin{proof}
Let $T$ be a reflexive bridge from $(\m a,\rho)$ to $(\m a,\sigma)$.  
By Lemma~\ref{lm:goodbridge}, there exists a reflexive bridge
$T'$ from
$(\m a,\rho,\rho^\ast)$ to $(\m a,\sigma,\sigma^\ast)$.
Now apply Lemma~\ref{lm:sameopt-gen}.
\end{proof}

This completes our development of the basic terminology and results concerning bridges
from \cite{zhuk2020}.

\section{\Rooted\ bridges in locally finite Taylor varieties} \label{sec:taylor}

In this short section we use tame congruence theory, first to characterize
irreducible congruences in finite Taylor algebras, and secondly to extend 
Definitions~\ref{df:optimal} and~\ref{df:adjacent} and Corollary~\ref{lm:sameopt} from irreducible congruences to meet-irreducible congruences in
finite Taylor algebras.
One tool we use is a restriction of the concept of ``bridge," which we call
``\rooted\ bridge."

\begin{df} \label{df:Cov}
Suppose $\m a$ is a finite algebra, $\rho \in \Con(\m a)$ is meet-irreducible, and
$\coverrho$ is its unique upper cover.  
$\Cov(\rho)$ denotes $\set{\tau \in \genCov(\rho)}{$\tau \subseteq \coverrho$}$.
\end{df}

If $\rho$ is irreducible, then clearly $\Cov(\rho)=\{\rho^\ast\}$.
However, it can happen that $\rho$ is meet-irreducible and $|\Cov(\rho)|=1$,
yet $\rho$ is not irreducible; see Lemma~\ref{lm:irred}.

\begin{df} \label{df:rooted}
Suppose $\m a$ and $\m b$ are finite algebras in a common signature,
$\rho \in \Con(\m a)$, $\sigma \in \Con(\m b)$,
and $\rho$ and $\sigma$ are meet-irreducible.  A bridge $T$ from $(\m a,\rho)$
to $(\m b,\sigma)$ is \emph{\rooted}\ if there exist $\tau \in \Cov(\rho)$
and $\tau' \in \Cov(\sigma)$ such that $T$ contains a bridge $T'$ from $(\m a,\rho,\tau)$ to
$(\m b,\sigma,\tau')$ with $\tr(T')=\tr(T)$.
\end{df}

Equivalently, a bridge $T$ from $(\m a,\rho)$ to $(\m b,\sigma)$
is \rooted\ if and only if the set 
\[
T_0:=T\cap \set{(a,a',b,b')}{$(a,a')\in\coverrho$ and
$(b,b')\in \coversigma$}
\]
satisfies $\proj_{1,2}(T_0)\ne \rho$,
where $\coverrho,\coversigma$ are the unique upper covers of $\rho,\sigma$
respectively.

Observe that if $\rho\in \Con(\m a)$ and $\sigma \in \Con(\m b)$ are irreducible, 
then every bridge
from $(\m a,\rho)$ to $(\m b,\sigma)$ is \rooted\ by Lemma~\ref{lm:goodbridge}.

Next, we introduce some notation and record two facts about
$\Cov(\rho)$ given by tame congruence theory.

\begin{df} \label{df:barrho}
Suppose $\m a$ is a finite algebra, $\rho \in \Con(\m a)$ is meet-irreducible, and
$\coverrho$ is its unique upper cover.  
Let $\barrho$ denote the following subset of $\coverrho$:
\[
\barrho = \rho \circ
\left(0_A \cup \bigcup\set{N^2}{$N$ is a $(\rho,\coverrho)$-trace}\right)\circ \rho.
\]
\end{df}

\begin{prp}[\mbox{\cite[Lemma 5.24]{tct}}] \label{prp:basictol}
Suppose $\m a$ is a finite algebra, $\rho \in \Con(\m a)$ is meet-irreducible, and
$\coverrho$ is its unique upper cover.  
\begin{enumerate}
\item \label{basictol:it1}
If $\typ(\rho,\coverrho)\in \{\tup 2,\tup 3\}$, then $\barrho \leq \m a^2$
and $\Cov(\rho) = \{\barrho\}$.
\item \label{basictol:it2}
If $\typ(\rho,\coverrho)\in \{\tup 4,\tup 5\}$, then 
$|\Cov(\rho)|=2$, say $\Cov(\rho)=\{\tau_0,\tau_1\}$.  Moreover,
$\taurho_1=\taurho_0^{-1}$, $\taurho_0\cap \taurho_1=\rho$, and $\taurho_0\cup \taurho_1=\barrho$.
 \end{enumerate}
\end{prp}

Now we can characterize irreducible congruences in finite Taylor algebras.

\begin{lm} \label{lm:irred}
Suppose $\m a$ is a finite Taylor algebra, 
$\rho \in \Con(\m a)$
is meet-irreducible, and $\coverrho$ is its unique upper cover.  
\begin{enumerate}
\item \label{irred:it2}
The following are equivalent:
\begin{enumerate}
\item \label{irred:it2a}
$\rho$ is irreducible.
\item \label{irred:it2b}
$\typ(\rho,\coverrho)\in \{\tup 2,\tup 3\}$, and for all $(a,b) \in A^2\setminus
\coverrho$ there exists a unary polynomial $f(x) \in \Pol_1(\m a)$ with
$(f(a),f(b)) \in \coverrho\setminus \rho$.
\end{enumerate}

\item\label{irred:it1}
If $\rho$ is irreducible, then $\rho^\ast=\barrho$.

\end{enumerate}
\end{lm}

\begin{proof}
Assume $\rho$ is irreducible.  
We have $\typ(\rho,\coverrho)\ne \tup 1$ by Theorem~\ref{thm:locfin}.
Irreducibility of $\rho$ forces $|\Cov(\rho)|=1$, so
$\typ(\rho,\coverrho)\in \{\tup 2,\tup 3\}$ by
Proposition~\ref{prp:basictol}\eqref{basictol:it2}.
Assume next that there exists $(a,b) \in A^2\setminus \coverrho$ such that
$(f(a),f(b)) \not\in \coverrho\setminus \rho$ for all $f \in \Pol_1(\m b)$.
Let $\sigma = \sg^{\m a^2}(\{(a,b)\} \cup 0_A)$.  
The assumption implies $\sigma \cap \coverrho \subseteq
\rho$, so $(\rho \circ \sigma \circ \rho) \cap \coverrho = \rho$, which 
would contradict irreducibility of $\rho$.  Thus if $\rho$ is irreducible,
then the conditions in item~\eqref{irred:it2b} hold.

Conversely, assume that the conditions in item~\eqref{irred:it2b} hold.
Since $\typ(\rho,\coverrho) \in \{\tup 2,\tup 3\}$, we have $\Cov(\rho)=\{\barrho\}$
by Proposition~\ref{prp:basictol}\eqref{basictol:it1}.
We will show that $\rho$ is irreducible with $\rho^\ast=\barrho$, which will
also establish item~\eqref{irred:it1}.
Let $R$ be a $\rho$-saturated subuniverse  of $\m a^2$
which properly contains $\rho$; we must show $\barrho\subseteq R$.
It will suffice to prove $R\cap \coverrho\ne \rho$, as then $R\cap \coverrho$ will be
a $\rho$-saturated subuniverse of $\m a^2$ satisfying $\rho \subset R\cap \coverrho\subseteq 
\coverrho$, so $\barrho\subseteq R\cap \coverrho$ as $\Cov(\rho)=\{\barrho\}$
by Proposition~\ref{prp:basictol}\eqref{basictol:it1}.
To prove $R\cap \coverrho \ne \rho$, pick $(a,b) \in R
\setminus \rho$.  If $(a,b) \in \coverrho$ then 
we are done, so assume $(a,b)\not\in \coverrho$.
By condition~\eqref{irred:it2b}, 
there exists $f \in \Pol_1(\m a)$ with $(a',b') :=(f(a),f(b)) \in \coverrho\setminus \rho$.  
Then $(a',b') \in \sg^{\m a^2}(\{(a,b)\}\cup 0_A) \subseteq R$,
so $(a',b')$ witnesses $R\cap \coverrho\ne \rho$.
\end{proof}

The following easy lemma will help us extend the notions of $\Opt(\rho)$ and adjacency from
irreducible congruences to meet-irreducible congruences.  

\begin{lm} \label{lm:type45}
Suppose $\m a$ is a finite Taylor algebra, 
$\rho\in \Con(\m a)$ is meet-irreducible, and
$\tau,\tau' \in \Cov(\rho)$. Then there exists a reflexive bridge from $(\m a,\rho,\tau)$
to $(\m a,\rho,\tau')$.
Hence $\Opt(\rho,\tau)=\Opt(\rho,\tau')$.
\end{lm}

\begin{proof}
If $\tau'=\tau$, then we can use the identity bridge $I_{(\m a,\rho,\tau)}$.  Otherwise,
by Proposition~\ref{prp:basictol} we must have $\typ(\rho,\coverrho)\in \{\tup 4,\tup 5\}$
and $\tau'=\tau^{-1}$.
Then the set $T=\set{(a_1,a_2,b_1,b_2)}{$(a_1,a_2,b_2,b_1) \in I_{(\m a,\rho,\tau)}$}$
is the required reflexive bridge.  The last claim follows by Lemma~\ref{lm:sameopt-gen}.
\end{proof}

Lemma~\ref{lm:type45} justifies the 
the following extension of the notation $\Opt(\rho)$ from irreducible congruences to
meet-irreducible congruences.

\begin{df}
Suppose $\m a$ is a finite Taylor algebra 
and $\rho\in \Con(\m a)$ is meet-irreducible.
\begin{enumerate}
\item
$\Opt(\rho)$ denotes the (unique) congruence $\Opt(\rho,\tau)$ where $\tau \in \Cov(\rho)$.
\item
A bridge $T$ from $(\m a,\rho)$ to $(\m a,\rho)$ is \emph{optimal} if
it is \rooted\ and $\tr(T)=\Opt(\rho)$.
\end{enumerate}
\end{df}

We also extend the adjacency relation to meet-irreducible congruences, as follows.

\begin{df} \label{df:adjacent2}
Suppose $\m a$ is a finite Taylor algebra 
and $\rho,\sigma\in \Con(\m a)$ are meet-irreducible.
Say that $\rho$ and $\sigma$ are \emph{adjacent} if there exists a reflexive 
\rooted\ bridge 
from $(\m a,\rho)$ to $(\m a,\sigma)$.
\end{df}

\begin{rem}
Definition~\ref{df:adjacent2} extends Definition~\ref{df:adjacent}, since
every bridge between irreducible congruences is \rooted.
Definition~\ref{df:adjacent2} disagrees with Zhuk's definition in 
\cite{zhuk2020} when $\rho,\sigma$ are not irreducible, as
Zhuk does not require the bridge to be \rooted.  
\end{rem}

Corollary~\ref{lm:sameopt} extends to meet-irreducible congruences.  The proof is a simple
application of Lemmas~\ref{lm:sameopt-gen} and~\ref{lm:type45}.

\begin{cor} \label{cor:sameopt}
Suppose $\m a$ is a finite Taylor algebra 
and $\rho,\sigma \in \Con(\m a)$ are meet-irreducible.  If $\rho$ and $\sigma$ are
adjacent, then $\Opt(\rho)=\Opt(\sigma)$.
\end{cor}

\section{Connecting bridges to centrality and similarity}
\label{sec:central}

In this final section we give our main results.  We prove that, 
in the context developed in the previous section, the $\Opt(\rho)$ construction
is simply the centralizer $(\rho:\coverrho)$ (Lemma~\ref{lm:opt2}); 
we apply this and results from \cite{similar} to extend two important
results from \cite{zhuk2020} (Theorem~\ref{thm:zeta} and Lemma~\ref{lm:adjacent});
and we prove that the ``there exists a \rooted\ bridge"
relation between pairs $(\m a,\rho)$ and $(\m b,\sigma)$
is exactly the similarity relation between the respective 
quotient algebras $\m a/\rho$ and $\m b/\sigma$ (Corollary~\ref{cor:zhuk-equiv}).

\begin{df} \label{df:TArho}
Suppose $\m a$ is a finite Taylor algebra 
and $\rho\in \Con(\m a)$ is meet-irreducible with unique upper cover $\coverrho$.
If $\coverrho/\rho$ is abelian, let $\alpha=(\rho:\coverrho)$, define
\[
\Delta^\flat_{\coverrho,\alpha} = \set{(a_1,a_2,b_1,b_2) \in A^4}{$
((a_1,a_2),(b_1,b_2)) \in \Delta_{\coverrho,\alpha}$},
\]
and set $\Topt_{(\m a,\rho)} = I_{(\m a,\rho,\coverrho)} \circ \Delta_{\coverrho,\alpha}^\flat
\circ I_{(\m a,\rho,\coverrho)}$.
\end{df}

\begin{lm} \label{lm:opt2}
Suppose $\m a$ is a finite Taylor algebra 
and $\rho \in \Con(\m a)$ is meet-irreducible with unique upper cover 
$\coverrho$.  
\begin{enumerate}
\item \label{opt2:it1}
$\Opt(\rho)=(\rho:\coverrho)$.  
\item \label{opt2:it2}
Hence $\coverrho/\rho$ is
abelian if and only if $\Opt(\rho)> \rho$.
\item \label{opt2:it3}
If $\coverrho/\rho$ is abelian, then $\barrho=\coverrho$, so $\Cov(\rho) = \{\coverrho\}$, and
$\Topt_{(\m a,\rho)}$ is an optimal bridge
from $(\m a,\rho)$ to itself.
\end{enumerate}
\end{lm}

\begin{proof}
We first show if $\coverrho/\rho$ is abelian, i.e., $\typ(\rho,\coverrho)=\tup 2$,
then $\barrho=\coverrho$.
Note that $\barrho/\rho$ is a reflexive subuniverse of $(\m a/\rho)^2$ contained in the abelian
minimal congruence $\coverrho/\rho$ and properly containing $0_{A/\rho}$.
$\m a/\rho$ has a weak difference term by Theorem~\ref{thm:locfin}, which is a Maltsev
operation when restricted to each block of $\coverrho/\rho$.
Hence $\barrho/\rho =\coverrho/\rho$ by Lemma~\ref{lm:maltsev}.  
Since $\barrho$ is $\rho$-saturated, it follows that $\barrho = \coverrho$.
Hence $\Cov(\rho)=\{\coverrho\}$ by Proposition~\ref{prp:basictol}\eqref{basictol:it1}.

Next we show that if $\coverrho/\rho$ is abelian, then $\Topt_{(\m a,\rho)}$ 
is a 
bridge from $(\m a,\rho,\coverrho)$ to itself with trace $\alpha:=(\rho:\coverrho)$.  
Note that $\Delta^\flat_{\coverrho,\alpha}$ is not necessarily a bridge
from $(\m a,\rho,\coverrho)$ to itself, 
because it may fail to satisfy \ref{pbridge:it0}.  
However it does satisfy \ref{pbridge:it1} and \ref{pbridge:it2},
the latter by Lemma~\ref{lm:Deltaprop} using $C(\alpha,\coverrho;\rho)$.
Pre- and post-composing $\Delta^\flat_{\coverrho,\alpha}$
with $I_{(\m a,\rho,\coverrho)}$ preserves
\ref{pbridge:it1} and \ref{pbridge:it2} and also guarantees \ref{pbridge:it0}.
Thus $\Topt_{(\m a,\rho)}$ is indeed a bridge from $(\m a,\rho,\coverrho)$ to itself.
Finally, 
\[
\tr(\Topt_{(\m a,\rho)}) = 
\tr(I_{(\m a,\rho,\coverrho)}) \circ
\tr(\Delta^\flat_{\coverrho,\alpha}) \circ
\tr(I_{(\m a,\rho,\coverrho)})  = \rho \circ \alpha \circ \rho = \alpha
\]
as required.  Clearly $\Topt_{(\m a,\rho)}$ is \rooted.

Next, we will prove \eqref{opt2:it1}.
Assume first that $\coverrho/\rho$ is nonabelian.  Because $\rho$ is 
meet-irreducible and $\coverrho$ is its unique upper cover, we get
$(\rho:\coverrho)=\rho$.
As $\Opt(\rho)\geq \rho$ by Lemma~\ref{lm:opt}, it will be enough in this case to
show that $\Opt(\rho)\ngeq \coverrho$.
Assume instead that $\Opt(\rho)\geq \coverrho$.
Choose $\tau \in \Cov(\rho)$ and let
$T_0$ be an optimal bridge from $(\m a,\rho,\tau)$ to itself.
Let $T=T_0\circ T_0$; then $T$ is also an optimal bridge from
$(\m a,\rho,\tau)$ to itself. Hence $\tr(T)=\Opt(\rho)$, and moreover
$(a,b,a,b) \in T$ for all $(a,b)\in \tau$.

Pick a $(\rho,\coverrho)$-trace $N$.
By tame congruence theory,
 i.e., Proposition~\ref{prp:tct},
there exists $(0,1) \in N^2\setminus \rho$ 
and a binary polynomial
$p(x,y) \in \Pol_2(\m a)$ such that $p(0,0)=p(0,1)=p(1,0)=0$ and $p(1,1)=1$.  
Then $(0,1) \in \barrho$ (see Definition~\ref{df:barrho}), and since $\barrho =\tau\cup \tau^{-1}$
by Proposition~\ref{prp:basictol}, we have either $(0,1) \in \tau$ or $(1,0)\in \tau$.
Assume with no loss of generality that $(0,1) \in \tau$.

Choose a
$(2+n)$-ary term and a tuple $\tup c \in A^n$ so that $p(x,y) = t(x,y,\tup c)$.
Then we have the following tuples in $T$:
\[
\begin{array}{ccl}
(~0,~1,~0,~1\,) \in T&  \rule{.5in}{0in} &\mbox{as $(0,1)\in\tau$}\\[.05in]
(~0,~0,~1,~1\,) \in T& & \mbox{as $(0,1) \in \coverrho \leq \Opt(\rho)=\tr(T)$}\\[.05in]
\begin{array}{c}
(c_1,c_1,c_1,c_1) \in T\\\vdots\\(c_n,c_n,c_n,c_n) \in T
\end{array} & \left. \rule{0in}{.35in}\right\}  & \mbox{as $T$ is reflexive} 
\end{array}
\]
Applying $t$ coordinate-wise to these tuples gives $(0,0,0,1)\in T$,
contradicting \ref{pbridge:it2} since $(0,1)\not\in \rho$.  This contradiction
proves $\Opt(\rho)\ngeq \coverrho$ and hence $\Opt(\rho)=\rho$ when $\coverrho/\rho$ is
nonabelian.

Assume next that $\coverrho/\rho$ is abelian.
We will first show $\Opt(\rho)\leq (\rho:\coverrho)$.
As shown earlier, $\Cov(\rho)=\{\coverrho\}$.
Fix an optimal bridge $T$ from $(\m a,\rho,\coverrho)$ to itself.
By replacing $T$ with $T\circ T$, we may assume that $(a,b,a,b) \in T$ for all $(a,b) \in
\coverrho$.  To prove $\Opt(\rho) \leq (\rho:\coverrho)$, we will simply show that
$C(\tr(T),\coverrho;\rho)$ holds by verifying the condition in
Definition~\ref{df:cent}\eqref{dfcent:it3}.

Let $t(x,\tup y)$ be a $(1+n)$-ary term, let $(a,b) \in \tr(T)$, and let
$(c_i,d_i) \in \coverrho$ for $i=1,\ldots,n$.
Then we have the following tuples in $T$:
\begin{eqnarray*}
&(~a,~a,~b,~b~)\in T&\\
&(c_1,d_1,c_1,d_1)\in T&\\
&\vdots&\\
&(c_n,d_n,c_n,d_n)\in T&
\end{eqnarray*}
Applying $t$ coordinatewise gives
\[
(t(a,\tup c),t(a,\tup d),t(b,\tup c),t(b,\tup d)) \in T.
\]
Then by \ref{pbridge:it2},
\[
t(a,\tup c) \stackrel{\rho}{\equiv} t(a,\tup d) \iff
t(b,\tup c) \stackrel{\rho}{\equiv} t(b,\tup d).
\]
This proves that $C(\tr(T),\coverrho;\rho)$ holds
and hence $\Opt(\rho)=\tr(T)\leq (\rho:\coverrho)$.

On the other hand, we have already shown that $\Topt_{(\m a,\rho)}$ is a 
\rooted\ bridge
from $(\m a,\rho)$ to itself with trace $(\rho:\coverrho)$.  
Hence $(\rho:\coverrho)\leq \Opt(\rho)$, which proves $\Opt(\rho)=(\rho:\coverrho)$ in the
abelian case and completes
the proof of \eqref{opt2:it1} and \eqref{opt2:it3}. \eqref{opt2:it2} follows
from \eqref{opt2:it1}.
\end{proof}

Using our results about similarity,
we can now easily obtain (and generalize) one of the
key results in Zhuk \cite{zhuk2020}.

\begin{df}
An algebra is \emph{affine} if it is abelian and has a Maltsev term.
\end{df}

\begin{thm}[Cf.\ Zhuk {\cite[Corollary 8.17.1]{zhuk2020}}] \label{thm:zeta}
Suppose $\m a$ is a finite Taylor algebra, 
$\rho \in \Con(\m a)$ is irreducible, and $\Opt(\rho)=1_A$.
Then there exists a simple affine algebra $\m z \in \HS(\m a^2)$ having a 1-element
subuniverse $\{0\}\leq \m z$, and there exists a subdirect subuniverse
$\zeta \leq_{sd} \m a\times \m a\times \m z$ with $\proj_{1,2}(\zeta)=\rho^\ast$,
such that for all $(a,a',b) \in \zeta$,
\[
(a,a') \in \rho \iff b=0.
\]
\end{thm}

\begin{proof}
Let $\coverrho$ be the unique upper cover of $\rho$.
By Lemma~\ref{lm:opt2}, we have $(\rho:\coverrho)=1_A$ and $\coverrho/\rho$ is abelian.
Thus $\rho^\ast=\coverrho$ by Lemmas~\ref{lm:irred}\eqref{irred:it1} 
and~\ref{lm:opt2}\eqref{opt2:it3}.  Let
$\overline{\m a} = \m a/\rho$ and $\mu = \coverrho/\rho$; thus $\overline{\m a}$ is
subdirectly irreducible with abelian monolith $\mu$ satisfying
$(0:\mu)=1$.  Let $\m z = D(\overline{\m a})$.
By Theorem~\ref{cor:D(A)}, $\m z$ is simple and abelian, so is affine, and has
a 1-element subuniverse $\Do = \{0\}$. 
Also by Theorem~\ref{cor:D(A)},
there exists a surjective homomorphism $h:\overline{\m a}(\mu) \ra \m z$ such that
$h^{-1}(0) = 0_{\barA}$.  Let
\[
\zeta = \set{(a,a',b) \in A\times A\times Z}{$(a,a') \in \coverrho$ and
$h((a/\rho,a'/\rho)) = b$}.
\]
$\zeta$ has the required properties.
\end{proof}

\begin{rem}
Zhuk \cite{zhuk2020} proved Theorem~\ref{thm:zeta} in the special
case where the signature of $\m a$ consists of just one operation, $w(x_1,\ldots,x_m)$,
which is an $m$-ary \emph{special WNU}, that is, an (idempotent) weak near-unanimity 
operation whose derived binary operation
$x\circ y := w(x,\ldots,x,y)$ satisfies $x\circ (x\circ y)=x\circ y$.
In this context, if $\m z \in \HS(\m a^2)$ then $w$ is also an $m$-ary special WNU
in $\m z$; if in addition $\m z$ is simple and affine, then it is not hard to show
(cf.\ \cite[Lemma 6.4]{zhuk-key}) that
$\m z\cong (\mathbb Z_p,x_1+\cdots +x_m \pmod p)$ for some prime $p$ which is a 
divisor of $m-1$.  Zhuk stated his \cite[Corollary~8.17.1]{zhuk2020} with this
stronger conclusion.
\end{rem}

Next we establish a simple invariant of \rooted\ bridges.

\begin{thm} \label{lm:newsametype}
Suppose $\m a,\m b$ are finite algebras in a locally finite Taylor variety,
and 
$\rho \in \Con(\m a)$, $\sigma \in \Con(\m b)$, where $\rho$ and $\sigma$
are meet-irreducible.  Let $\coverrho,\coversigma$ be the respective unique upper
covers of $\rho,\sigma$.  
Assume that there exists a \rooted\ bridge from $(\m a,\rho)$ to $(\m b,\sigma)$.
Then 
$\coverrho/\rho$ is abelian if and only if $\coversigma/\sigma$ is abelian.
\end{thm}

\begin{proof}
Let $T$ be a \rooted\ bridge from $(\m a,\rho)$ to $(\m b,\sigma)$.  
We may assume that $T$ is from $(\m a,\rho,\tau)$ to $(\m b,\sigma,\tau')$
where $\tau \in \Cov(\rho)$ and $\tau' \in \Cov(\sigma)$.
Assume for the sake of contradiction that $\coverrho/\rho$ is abelian while
$\coversigma/\sigma$ is not.  Then by tame congruence theory,
$\typ(\rho,\coverrho)=\tup 2$ while $\typ(\sigma,\coversigma)\in \{\tup 3,\tup 4,
\tup 5\}$.  Hence $\tau = \coverrho$ by Lemma~\ref{lm:opt2}
and $\tau' \cup (\tau')^{-1} = \barsigma$ by Proposition~\ref{prp:basictol}.

By passing to $\overline{\m a}:=\m a/\rho$ and $\overline{\m b}:= \m b/\sigma$,
we may assume that $\rho=0_A$ and $\sigma=0_B$.
Rename $\coverrho$ and $\coversigma$ as $\mu_{\m a}$ and $\mu_{\m b}$ respectively.
Let $\alpha = (0_A:\mu_{\m a})$ and observe that $(0_B:\mu_{\m b})= 0_B$.  Replace $T$
with $\Topt_{(\m a,0)} \circ T$; 
then by Lemma~\ref{lm:modiso}, the rule
\[
h(a)=b \iff (a,b) \in \tr(T)
\]
defines a surjective homomorphism $h:\m a\ra \m b$ with kernel $\alpha$.

\begin{clm} \label{clm:rBtoA}
For all $f \in \Pol_k(\m b)$ there exists $f_{\m a} \in
\Pol_k(\m a)$ such that 
\begin{enumerate}
\item
$h(f_{\m a}(x_1,\ldots,x_k)) = f(h(x_1),\ldots,h(x_k))$ for all $x_1,\ldots,x_k \in A$.
\item \label{rBtoA:it2}
For all $(x_1,y_1,u_1,v_1),\ldots,(x_k,y_k,u_k,v_k) \in T$ we have
\[
(f_{\m a}(\tup x),f_{\m a}(\tup y),f(\tup u),f(\tup v)) \in T.
\]
\end{enumerate}
\end{clm}

\begin{proof}[Proof of Claim~\ref{clm:rBtoA}]
Indeed, if we select a term $t(x_1,\ldots,x_k,y_1,\ldots,y_n)$ and $\tup{b} \in B^n$
so that $f(\tup{x})=t^{\m b}(\tup{x},\tup{b})$, then we simply need to select
$\tup a \in A^n$ with $h(a_i)=b_i$ for each $i \in [n]$ and then define
$f_{\m a}(\tup{x}) = t^{\m a}(\tup{x},\tup{a})$.  Item (1) then follows immediately, 
and item (2) follows from the fact that $(a_i,a_i,b_i,b_i) \in T$ for all $i \in [n]$.
\end{proof}

By tame congruence theory, i.e., Proposition~\ref{prp:tct},
there exists a $(0_B,\mu_{\m b})$-minimal set $U$ with unique
$(0_B,\mu_{\m b})$-trace $N=\{0,1\}= U\cap C$, 
a unary polynomial $e(x) \in \Pol_1(\m b)$, and
a binary polynomial $p(x,y) \in \Pol_2(\m b)$, satisfying:
\begin{enumerate}
\item
$e(A)=U$ and $e(e(x))=e(x)$ for all $x \in A$.
\item
$p(x,x)=p(x,1)=p(1,x) = x$ for all $x \in U$.
\item
$p(x,0)=p(0,x)=x$ for all $x \in U\setminus \{1\}$.
\end{enumerate}

Because $T$ is \rooted, we have $(0,1) \in \proj_{3,4}(T)$ or $(1,0)\in \proj_{3,4}(T)$.
Assume with no loss of generality that $(0,1) \in \proj_{3,4}(T)$.
Pick $(a',b')\in \mu_{\m a}$ with $(a',b',0,1) \in T$.
Let $e_{\m a} \in \Pol_1(\m a)$ be a polynomial given by Claim~\ref{clm:rBtoA} for $e(x)$.
Let $a = e_{\m a}(a')$ and $b = e_{\m a}(b')$.
Then $(a,b,0,1) \in T$ and $h(a)=e(h(a'))\in U$ by Claim~\ref{clm:rBtoA}.
Since $(a,b) \in \mu_{\m a} \subseteq\ker(h)$ we get $h(b)=h(a)$.  
Let $u = h(a) \in U$.  

Now let $p_{\m a} \in \Pol_2(\m a)$ be a polynomial given by Claim~\ref{clm:rBtoA}
for $p(x,y)$.  In calculations that follow, we will denote both
$p(x,y)$ and $p_{\m a}(x,y)$ by $xy$.

We have the following elements of $T$:
\begin{align*}
\tau_1=(x_1,y_1,u_1,v_1) :={}& (a,b,0,1)\\
\tau_2=(x_2,y_2,u_2,v_2) :={}& (a,a,u,u)\\
\tau_3=(x_3,y_3,u_3,v_3) :={}& (b,b,u,u).
\end{align*}

\noindent\textsc{Case 1}: $u\ne 1$.

Applying Claim~\ref{clm:rBtoA}\eqref{rBtoA:it2} to $p(x,y)$ and the pairs $(\tau_1,\tau_1)$, $(\tau_2,\tau_1)$ and
$(\tau_1,\tau_3)$ respectively gives
\begin{align}
(aa,bb,0,1)&\in T \label{clm:inT1}\\
(aa,ab,u,u) &\in T\label{clm:inT2}\\
(ab,bb,u,u) &\in T. \label{clm:inT3}
\end{align}
\eqref{clm:inT1} with \ref{pbridge:it2} gives
$aa\ne bb$, but \eqref{clm:inT2} and \eqref{clm:inT3} with \ref{pbridge:it2}
give $aa=ab=bb$, contradiction.  

\medskip
\noindent\textsc{Case 2}: $u=1$.

Applying Claim~\ref{clm:rBtoA}\eqref{rBtoA:it2} to $p(x,y)$ and the pairs $(\tau_1,\tau_1)$, $(\tau_2,\tau_1)$ and
$(\tau_1,\tau_3)$ respectively gives
\begin{align}
\sigma_1:= (aa,bb,0,1)&\in T \label{clm:inT4}\\
\sigma_2:=(aa,ab,0,1) &\in T\label{clm:inT5}\\
\sigma_3:=(ab,bb,0,1) &\in T. \label{clm:inT6}
\end{align}
Let $d(x,y,z)$ be a weak difference term for the locally finite Taylor variety
containing $\m a$ and $\m b$.
Recall that $\mu_{\m a}$ is abelian and observe that
$aa,ab,bb$ all belong to a common $\mu_{\m a}$-class.
Applying $d$ coordinate-wise to the tuples $\sigma_2,\sigma_1,\sigma_3$ (in that order)
and using the defining property of weak difference terms, 
we get
$(ab,ab,0,1) \in T$, which again contradicts \ref{pbridge:it2}.

As we have found a contradiction in both cases, the theorem is proved.
\end{proof}

Now we can extend and give a relatively short proof of
an important result about adjacent congruences in \cite{zhuk2020}.

\begin{lm} [cf.\ Zhuk {\cite[Lemma 8.18]{zhuk2020}}] \label{lm:adjacent}
Suppose $\m a$ is a finite Taylor algebra 
and $\rho,\sigma \in \Con(\m a)$ are meet-irreducible and adjacent.
If $\rho\ne \sigma$, then $\Opt(\sigma)>\sigma$.
\end{lm}

\begin{proof}
Let $\coverrho,\coversigma$ be the unique upper covers of $\rho,\sigma$ respectively.
Assume $\Opt(\sigma)=\sigma$.  Then $\coversigma/\sigma$ is nonabelian by 
Lemma~\ref{lm:opt2}\eqref{opt2:it2}.  Hence $\coverrho/\rho$ is nonabelian by 
Theorem~\ref{lm:newsametype}, so $\rho=\Opt(\rho) = \Opt(\sigma)=\sigma$ by
Lemma~\ref{lm:opt2}\eqref{opt2:it2} and Corollary~\ref{lm:sameopt}, 
contradicting $\rho\ne\sigma$.
\end{proof}

For the remainder of this section, we work to characterize the ``there exists a 
\rooted\ bridge"
relation between meet-irreducible congruences.  The next lemma handles the nonabelian case.

\begin{lm} \label{lm:sametype}
Suppose $\m a,\m b$ are finite algebras in a locally finite Taylor variety
and 
$\rho \in \Con(\m a)$, $\sigma \in \Con(\m b)$, where $\rho$ and $\sigma$
are meet-irreducible.  Let $\coverrho,\coversigma$ be the respective unique upper
covers of $\rho,\sigma$.  
\begin{enumerate}
\item \label{sametype:it2}
If $\m a/\rho\cong \m b/\sigma$, then there exists a \rooted\ bridge from $(\m a,\rho)$ to
$(\m b,\sigma)$.
\item \label{sametype:it1}
Conversely, if
$\coverrho/\rho$ and $\coversigma/\sigma$ are both nonabelian and there
exists a \rooted\ bridge from $(\m a,\rho)$ to $(\m b,\sigma)$, then
$\m a/\rho\cong \m b/\sigma$.
\end{enumerate}
\end{lm}

\begin{proof}
\eqref{sametype:it2}
Suppose $\gamma:\m a/\rho\cong \m b/\sigma$ is an isomorphism.
Then
\[
T = \set{(a,a',b,b')\in A\times A\times B\times B}{$\gamma(a/\rho)=b/\sigma$ and
$\gamma(a'/\rho)=b'/\sigma$}
\]
is a \rooted\ bridge from $(\m a,\rho)$ to $(\m b,\sigma)$.

\eqref{sametype:it1}
We have $\Opt(\rho)=\rho$ and $\Opt(\sigma)=\sigma$ by Lemma~\ref{lm:opt2}\eqref{opt2:it2}.
Let $T$ be a \rooted\ bridge from $(\m a,\rho)$ to $(\m b,\sigma)$.  We can assume 
that $T$ is a bridge from $(\m a,\rho,\tau)$ to $(\m b,\sigma,\tau')$ for some
$\tau \in \Cov(\rho)$ and $\tau' \in \Cov(\sigma)$.
Thus $\rho=\Opt(\rho,\tau)$ and $\sigma = \Opt(\sigma,\tau')$.
Now the claim follows from Lemma~\ref{lm:modiso}.
\end{proof}

It remains to characterize the ``there exists a \rooted\ bridge" relation between 
meet-irreducible
congruences $\rho$ and $\sigma$ when $\coverrho/\rho$ and $\coversigma/\sigma$ are both
abelian.
We will see that there is a tight relationship to similarity and
\proper\ bridges as defined in \cite{similar}.
The main difficulty is that,
although Zhuk's definition of bridges is similar to the definition
of \proper\ bridges in \cite{similar}, the definitions 
differ in one essential way: we required \proper\ bridges
to satisfy
\begin{enumerate}[label=(B\arabic*)]
\setcounter{enumi}{2}
\item \label{pbridge:it3}
For all $(a_1,a_2,b_1,b_2) \in T$ we have $(a_i,b_i)\in \tr(T)$ for $i=1,2$,
\end{enumerate}
while Zhuk's bridges are not required to satisfy this condition.  Happily,
\rooted\ bridges between meet-irreducible congruences with abelian upper covers
can be assumed without loss of generality to satisfy \ref{pbridge:it3},
as we will prove in Theorem~\ref{thm:samebridge}.
First, we need the following result about the optimal bridges from Lemma~\ref{lm:opt2}\eqref{opt2:it3}.

\begin{lm} \label{lm:TDTDinv}
Suppose $\m a$ is a finite subdirectly Taylor algebra with abelian monolith $\mu$.
Let $\alpha = (0:\mu)$ and 
define $\Delta^\flat_{\mu,\alpha}$ as in Definition~\ref{df:TArho} (setting
$\rho:=0$, so $\coverrho=\mu$).
Recall the \proper\ bridge
$\TD a$ from $\m a$ to $D(\m a)$ 
defined in Corollary~\ref{cor:simprop},
and the optimal bridge $\Topt_{(\m a,0)}$ from $(\m a,0)$ to itself
defined in Definition~\ref{df:TArho}.  Then
$\Topt_{(\m a,0)} = \Delta^\flat_{\mu,\alpha} = \TD a \circ (\TD a)^\cup$.
\end{lm}

\begin{proof}
The first equality follows from the fact that
$I_{(\m a,0,\mu)} = \set{(a,b,a,b)}{$(a,b) \in \mu$}$.  
Let $\Delta = \Delta_{\mu,\alpha}$.
The second
equality will follow if we can show that for all $(a_1,a_2)$, $(b_1,b_2) \in
\mu$  with $a_1,a_2,b_1,b_2$ belonging to a common $\alpha$-class,
\[
(a_1,a_2)\stackrel{\Delta}{\equiv} (b_1,b_2) \iff 
\exists e \stackrel{\mu}{\equiv} a_1,
\exists e' \stackrel{\mu}{\equiv} b_1
\left( 
(a_1,e) \stackrel{\Delta}{\equiv} (b_1,e') ~\&~
(a_2,e) \stackrel{\Delta}{\equiv} (b_2,e')\right).
\]
The forward implication is easy: choose $e=a_2$ and $e' = b_2$.
For the reverse implication, 
apply a weak difference term component-wise to
\[
(a_1,e) \stackrel\Delta\equiv (b_1,e'),\quad
(a_2,e) \stackrel\Delta\equiv (b_2,e'),\quad
(a_2,a_2) \stackrel\Delta\equiv (b_2,b_2)
\]
to get $(a_1,a_2)\stackrel\Delta\equiv (b_1,b_2)$.
\end{proof}

\begin{thm} \label{thm:samebridge}
Suppose $\m a,\m b$ are finite algebras in a locally finite Taylor variety,
$\rho \in \Con(\m a)$ and $\sigma \in \Con(\m b)$ where both $\rho$ and $\sigma$
are meet-irreducible, and $\coverrho,\coversigma$ are their respective 
unique upper covers.  Assume that $\coverrho/\rho$ and $\coversigma/\sigma$
are abelian.
Then for every bridge $T$ from $(\m a,\rho,\coverrho)$ to
$(\m b,\sigma,\coversigma)$, the bridge
$T':= \Topt_{(\m a,\rho)}\circ T \circ \Topt_{(\m b,\sigma)}$
contains a bridge from $(\m a,\rho,\coverrho)$ to
$(\m b,\sigma,\coversigma)$ with the same trace as $T'$ and satisfying \ref{pbridge:it3}.
\end{thm}

\begin{proof}
First, we can assume with no loss of generality that $T'=T$.
Second, we can assume that $\rho=0_A$ and
$\sigma=0_B$.  For we can let $\overline{\m a}=\m a/\rho$,
$\overline{\m b}=\m b/\sigma$, $\mu=\coverrho/\rho$, $\montwo=\coversigma/\sigma$,
and 
\[
\barT = \set{(a_1/\rho,a_2/\rho,b_1/\sigma,b_2/\sigma)}{$(a_1,a_2,b_1,b_2) \in T$}
\]
and $\barT$ will be a bridge from $(\overline{\m a},0,\mu)$ to
$(\overline{\m b},0,\montwo)$ satisfying $\barT=T_{(\overline{\m a},0)}\circ \barT
\circ T_{(\overline{\m b},0)}$.
If there exists a bridge $T_1$ from $(\overline{\m a},0,\mu)$ to
$(\overline{\m b},0,\montwo)$ satisfying $T_1 \subseteq \barT$, 
$\tr(T_1)=\tr(\barT)$ and \ref{pbridge:it3}, then 
$T_0 := \set{(a_1,a_2,b_1,b_2)}{$(a_1/\rho,a_2/\rho,b_1/\sigma,b_2/\sigma)
\in T_1$}$ will be a bridge from $(\m a,\rho,\coverrho)$ to
$(\m b,\sigma,\coversigma)$ satisfying $T_0\subseteq T$, $\tr(T_0)=\tr(T)$ and \ref{pbridge:it3}.

So for the remainder of this proof assume that $T'=T$, $\rho=0_A$, and
$\sigma=0_B$.  For readability, rename $\coverrho$ as $\mu$ and $\coversigma$
as $\montwo$.  
Let $\alpha = (0:\mu)$, $\Delta_{\m a}=\Delta_{\mu,\alpha}$,
and $\Dmon_{\m a}=\baralpha/\Delta_{\m a}$.
Recall from
Corollary~\ref{cor:simprop} that the set
\[
\TD a = \set{(a_1,a_2,(a_1,e)/\Delta_{\m a},(a_2,e)/\Delta_{\m a})
\in A\times A\times D(\m a) \times D(\m a)}
{$a_1\stackrel{\mu}{\equiv} a_2\stackrel{\mu}{\equiv} e$}
\]
is a \proper\ bridge from $\m a$ to $D(\m a)$, and hence is
a bridge from $(\m a,0,\mu)$ to $(D(\m a),0,\Dmon_{\m a})$ satisfying \ref{pbridge:it3}.

Similarly define $\beta=(0:\montwo)$, $\Delta_{\m b}=\Delta_{\montwo,\beta}$, and
$\Dmon_{\m b}=\barbeta/\Delta_{\m b}$; then
\[
\TD b = \set{(b_1,b_2,(b_1,u)/\Delta_{\m b},(b_2,u)/\Delta_{\m b})
\in B\times B\times D(\m b) \times D(\m b)}
{$b_1\stackrel{\mu}{\equiv} b_2\stackrel{\mu}{\equiv} u$}
\]
is a bridge from $(\m b,0,\montwo)$ to $(D(\m b),0,\Dmon_{\m b})$ satisfying \ref{pbridge:it3}.
Thus by composing, we get the bridge
$T^\ast := (\TD a)^\cup \circ T \circ \TD b$
from $(D(\m a),0,\Dmon_{\m a})$ to $(D(\m b),0,\Dmon_{\m b})$.

Suppose there exists a bridge $T_1^\ast$ 
from $(D(\m a),0,\Dmon_{\m a})$ to $(D(\m b),0,\Dmon_{\m b})$ satisfying $T_1^\ast
\subseteq T^\ast$, $\tr(T_1^\ast)=\tr(T^\ast)$ and \ref{pbridge:it3}.
In this case we could define $T_1 = \TD a \circ T_1^\ast
\cup (\TD b)^\cup$.  Then $T_1$ will be a bridge
from $(\m a,0,\mu)$ to $(\m b,0,\montwo)$.
We will have
\begin{align*}
T_1 &\subseteq \TD a \circ T^\ast \circ (\TD b)^\cup\\
& = \TD a \circ ((\TD a)^\cup \circ T \circ \TD b) \circ (\TD b)^\cup\\
&= (\TD a\circ (\TD a)^\cup) \circ T \circ (\TD b\circ (\TD b)^\cup)\\
&= \Topt_{(\m a,0)} \circ T \circ \Topt_{(\m b,0)} &\mbox{by Lemma~\ref{lm:TDTDinv}}\\
&= T &\mbox{as we've assumed $T'=T$.}
\end{align*}
Hence $\tr(T_1)\subseteq \tr(T)$.  Similarly,
\begin{align*}
\tr(T_1) &= 
\tr(\TD a) \circ \tr(T_1^\ast) \circ \tr((\TD b)^\cup)\\
&=\tr(\TD a) \circ \tr(T^\ast) \circ \tr((\TD b)^\cup) &\tr(T_1^\ast)=\tr(T^\ast)\\
&= \tr(\Topt_{(\m a,0)})\circ \tr(T)\circ \tr(\Topt_{(\m b,0)})\\
&= \alpha \circ \tr(T) \circ \beta,
\end{align*} 
which proves $\tr(T)\subseteq \tr(T_1)$.  Hence $\tr(T_1)=\tr(T)$.
Finally, it is easy to check that $T_1$ satisfies \ref{pbridge:it3}, since each
of $\TD a,T_1^\ast,\TD b$ satisfies \ref{pbridge:it3}.

The remarks in the previous paragraph
serve to further reduce the proof of Theorem~\ref{thm:samebridge} to the case
where $\m a$ and $\m b$ are replaced by $D(\m a)$ and $D(\m b)$ respectively
(and $\rho=0_{D(A)}$ and $\sigma=0_{D(B)}$ and $T'=T$).
Put differently, in proving Theorem~\ref{thm:samebridge}, 
we can further assume with no loss of generality that
$\m a \cong D(\m a_1)$ and $\m b\cong D(\m b_1)$ for some subdirectly
irreducible algebras $\m a_1,\m b_1$ with abelian monoliths.  
It follows from this assumption and Theorem~\ref{cor:D(A)} that
$\alpha=\mu$, $\beta=\montwo$, and there exist subuniverses
$S_{\m a}\leq \m a$ and $S_{\m b}\leq \m b$ which are transversals for
$\mu$ and $\montwo$ respectively.

Recall that we are assuming $T=\Topt_{(\m a,0)}\circ T\circ \Topt_{(\m b,0)}$.  
In this context
this means $T=\Delta^\flat_{\mu,\mu} \circ T\circ \Delta^\flat_{\montwo,\montwo}$.
Hence $\tr(T) = \tr(\Delta^\flat_{\mu,\mu}) \circ \tr(T) \circ \tr(\Delta^\flat_{\montwo,\montwo})
= \mu\circ\tr(T)\circ\montwo$.

Recall from Lemma~\ref{lm:modiso} that $\tr(T)$ induces
an isomorphism $\gamma:\m a/\mu\cong \m b/\montwo$ defined by
\[
\gamma(a/\mu)=b/\montwo \iff (a,b) \in \mu\circ \tr(T) \circ \montwo = \tr(T).
\]
We also have the homomorphism $\pi_{\m a}:\m a\ra \m s_{\m a}$ 
which sends each $a \in A$ to
the unique element of $S_{\m a} \cap a/\mu$.  Likewise we have $\pi_{\m b}:\m b\ra \m s_{\m b}$.
These retractions naturally induce isomorphisms $\barpi_{\m a}:\m a/\mu\cong \m s_{\m a}$ and
$\barpi_{\m b}:\m b/\montwo \cong \m s_{\m b}$ given by $\barpi_{\m a}(a/\mu)=\pi_{\m a}(a)$
and $\barpi_{\m b}(b/\montwo)=\pi_{\m b}(b)$.  Thus we get an isomorphism $\delta:\m s_{\m a}
\cong \m s_{\m b}$ given by $\delta = \barpi_{\m b} \circ \gamma \circ (\barpi_{\m a})^{-1}$.
Equivalently, 
\begin{equation} \label{eq:graphdelta}
\graph(\delta)= \tr(T)\cap (S_{\m a}\times S_{\m b}).
\end{equation}

\begin{clm} \label{samebridge:clm1}
If there exists $(a_1,a_2,b_1,b_2) \in T$ with $a_1\ne a_2$ and $(a_1,b_1)\in \tr(T)$,
then $T$ contains a bridge $T_1$ from $(\m a,0,\mu)$ to $(\m b,0,\montwo)$
satisfying $\tr(T_1)=\tr(T)$ and \ref{pbridge:it3}.
\end{clm}

\begin{proof}[Proof of Claim~\ref{samebridge:clm1}]
Let $(a_1,a_2,b_1,b_2)\in T$ satisfy $a_1\ne b_1$ and $(a_1,b_1)\in \tr(T)$.
Observe that we also have $(a_2,b_2)\in \tr(T)$ as $\tr(T) = \mu \circ \tr(T)\circ \montwo$,
and $b_1\ne b_2$ by \ref{pbridge:it1}.
Let $T_1$ be the subuniverse of $\m a\times \m a\times \m b\times \m b$ generated by
\[
\set{(a,a,b,b)}{$(a,b) \in \tr(T)$} \cup \{(a_1,a_2,b_1,b_2)\}.
\]
Then $T_1\subseteq T$.
We will verify that $T_1$ is a bridge from $(\m a,0,\mu)$ to $(\m b,0,\montwo)$.
Property \ref{pbridge:it0} is trivially true and \ref{pbridge:it2} is inherited from $T$, so what
must be shown is \ref{pbridge:it1}: that $\proj_{1,2}(T_1)=\mu$ and $\proj_{3,4}(T_1)=\montwo$.
By construction, $\proj_{1,2}(T_1)$ is a reflexive subuniverse of $\mu$ properly containing $0_A$;
hence $\proj_{1,2}(T_1)=\mu$ by Lemma~\ref{lm:maltsev}.  A similar argument gives
$\proj_{3,4}(T_1)=\montwo$.  Thus $T_1$ is a bridge from $(\m a,0,\mu)$ to $(\m b,0,\montwo)$
satisfying $T_1\subseteq T$, and clearly $\tr(T_1)=\tr(T)$ by construction.

Now let
\[
W = \set{(a,a',b,b') \in A\times A\times B\times B}{
$(a,b),(a',b') \in \tr(T)$}.
\]
Note that $W$ is a subuniverse of $\m a\times \m a\times \m b\times \m b$.
As the generators of $T_1$ are contained in $W$, we get $T_1\subseteq W$, which proves that
$T_1$ satisfies \ref{pbridge:it3}.
\end{proof}

The remainder of the proof of Theorem~\ref{thm:samebridge} will consist
of the construction of a tuple $(a_1,a_2,b_1,b_2)\in T$
satisfying the hypotheses of Claim~\ref{samebridge:clm1}.
Define $R\leq \m a\times \m b$ by
\[
R = 
\proj_{1,3}(T\cap (A\times S_{\m a} \times B\times S_{\m b})).
\]

\begin{clm} \label{samebridge:clm2}
{~}
\begin{enumerate}
\item \label{sb-clm2:it1}
$\graph(\delta)\subseteq R$.  Thus $S_{\m a}\subseteq \proj_1(R)$ and $S_{\m b}\subseteq \proj_2(R)$.
\item \label{sb-clm2:it2}
For all $(a,b)\in R$ we have $a \in S_{\m a} \iff b \in S_{\m b}\iff (a,b)\in \graph(\delta)$.
\end{enumerate}
\end{clm}

\begin{proof}
\eqref{sb-clm2:it1} Assume $(a,b) \in \graph(\delta)$. Then
 $(a,b)\in\tr(T)\cap (S_{\m a}\times S_{\m b})$ by \eqref{eq:graphdelta},
so $(a,a,b,b) \in T \cap (S_{\m a}\times S_{\m a}\times S_{\m b}\times
S_{\m b})$, proving $(a,b) \in R$.

\eqref{sb-clm2:it2}
Suppose we have $(a,b) \in R$ with $a \in S_{\m a}$.  Choose $x \in S_{\m a}$ and
$y \in S_{\m b}$ with $(a,x,y,b)\in T$.  As $(a,x)\in \mu$ and $a,x \in S_{\m a}$, we
get $a=x$, which forces $b=y$ by bridge property \ref{pbridge:it2}.  Hence 
$(a,b) \in \tr(T)\cap (S_{\m a}\times S_{\m b})=\graph(\delta)$ by
\eqref{eq:graphdelta}.
\end{proof}

\begin{clm} \label{samebridge:clm3}
$\proj_1(R)\ne S_{\m a}$ and $\proj_2(R)\ne S_{\m b}$.
\end{clm}

\begin{proof}[Proof of Claim~\ref{samebridge:clm3}]
Pick a $\mu$-class $C$ with $|C|>1$.  Let $u$ be the unique element in $C\cap S_{\m a}$ and
pick $a \in C\setminus \{u\}$.  As $T$ is a bridge, there exists $(b,c) \in \montwo$
with $b\ne c$ and $(a,u,b,c) \in T$.  Let $y = \pi_{\m b}(b) \in S_{\m b}$ and let $x = \delta^{-1}(y) \in S_{\m a}$.  Then $(x,y) \in \graph(\delta)\subseteq \tr(T)$, and hence $(x,c) \in \tr(T)$
as well (using $\tr(T)=\tr(T)\circ\montwo$).

Thus the following are elements of $T$:
\[
(x,x,y,y),\quad (x,x,c,c),\quad (a,u,b,c).
\]
Applying the weak difference term $d(x,y,z)$ to these elements of $T$ gives
\begin{equation} \label{samebridge:eq1}
(a',u',b',y) \in T
\end{equation}
where
$a' = d(x,x,a)$, 
$u' = d(x,x,u)$ and
$b' = d(y,c,b)$.
Observe that $u' \in S_{\m a}$ since $x,u \in S_{\m a}$.
Thus \eqref{samebridge:eq1} implies $(a',b') \in R$.
Also note that $b,c,y$ belong to a common $\montwo$-class, so 
$b' = y-c+b$ calculated in the abelian group
$\Grp b(\montwo,y)$ by Lemma~\ref{lm:gumm}, so
$b'\ne y$ (since $b\ne c$).
Thus $b'\not\in S_{\m b}$, proving $\proj_2(R)\ne S_{\m b}$.  A similar proof gives
$\proj_1(R)\ne S_{\m a}$.
\end{proof}

\begin{clm} \label{samebridge:clm4}
$\proj_1(R)=A$ and $\proj_2(R)=B$.
\end{clm}

\begin{proof}[Proof of Claim~\ref{samebridge:clm4}]
By Claim~\ref{samebridge:clm2}, $S_{\m a}\subseteq \proj_1(R)$.
By Claim~\ref{samebridge:clm3}, 
we can choose $a_0 \in \proj_1(R)$ with $a_0\not\in S_{\m a}$.  Let $a \in A$ be arbitrary.
By Lemma~\ref{lm:poly}, there exists a term $t(x,\tup y)$ and a tuple $\tup u$ of elements
from $S_{\m a}$ such that $t(a_0,\tup u) = a$.  As $\proj_1(R)$ is a subuniverse, this proves
$a \in \proj_1(R)$, and as $a$ was arbitrary, we have shown $\proj_1(R)=A$.  
A similar proof gives $\proj_2(R)=B$.
\end{proof}

Now define 
\[
\theta := \set{(u,u') \in (S_{\m a})^2}{$\exists (a,b) \in R$ with $\pi_{\m a}(a)=u$ and
$\pi_{\m b}(b) = \delta(u')$}.
\]
If $b,c$ are two elements from $A$ or two elements from $B$, we will say that
$(b,c)$ is a \emph{Maltsev pair} if $d(b,b,c)=c=d(c,b,b)$.

\begin{clm} \label{samebridge:clm5}
{~}
\begin{enumerate}
\item \label{sb-clm5:it0}
For all $(a,u,b,y) \in T\cap (A\times S_{\m a}\times B\times S_{\m b})$,
if $y' = \delta(u)$ and $u' = \delta^{-1}(y)$, then
$(u,u'),(a,u'),(b,y'),(y,y')$ are Maltsev pairs.

\item \label{sb-clm5:it1}
For all $(u,u') \in \theta$, $(u,u')$ and $(u',u)$ are Maltsev pairs.
\item \label{sb-clm5:it3}
$\theta \in \Con(\m s_{\m a})$.
\end{enumerate}
\end{clm}

\begin{proof}[Proof of Claim~\ref{samebridge:clm5}]
\eqref{sb-clm5:it0}
Suppose $(a,u,b,y) \in T\cap(A\times S_{\m a}\times B\times S_{\m b})$
and $y'=\delta(u) \in S_{\m b}$ and $u'=\delta^{-1}(y) \in S_{\m a}$.
As $\graph(\delta)\subseteq\tr(T)$ 
by \eqref{eq:graphdelta}, we get $(u',u',y,y) \in T$.  
Let $x=d(u,u,u') \in S_{\m a}$.
Applying the weak difference term component-wise to 
$(a,u,b,y)$, $(a,u,b,y)$ and $(u',u',y,y)$ gives
\[
(d(a,a,u'),x,y,y)\in T.
\]
By bridge property \ref{pbridge:it2}, we get $d(a,a,u')=x$ and thus
$(x,y)\in \tr(T)$.  As $(x,y) \in S_{\m a}\times S_{\m b}$, we then get
$(x,y) \in \graph(\delta)$ by \eqref{eq:graphdelta}, so $x=u'$, which proves
$d(a,a,u')=d(u,u,u')=u'$.  The other required equalities are proved 
similarly.

\eqref{sb-clm5:it1}
Given $(u,u')\in \theta$, pick $(a,b) \in R$ with $\pi_{\m a}(a)=u$ and
$y:=\pi_{\m b}(b) =\delta(u')$.  
Because $(a,b) \in R$, we then get $(a,u,b,y)\in T$.  
Thus $(u,u')$ is a Maltsev pair by \eqref{sb-clm5:it0}.  
Let $y' = \delta(u)$; then $(y,y')$ is a Maltsev pair, again
by \eqref{sb-clm5:it0}.  Since $(y,y') = (\delta(u'),\delta(u))$ and
$\delta$ is an isomorphism, it follows that $(u',u)$ is a Maltsev pair.

\eqref{sb-clm5:it3}
$\theta$ is a subuniverse of $(\m s_{\m a})^2$ by virtue of how $\theta$ is 
defined, and is reflexive by Claim~\ref{samebridge:clm2}.
It then follows from \eqref{sb-clm5:it1}
and Lemma~\ref{lm:maltsev} that $\theta$ is a congruence.
\end{proof}

\begin{clm} \label{samebridge:clm7}
$R$ is the graph of an isomorphism $g:\m a\cong \m b$ extending $\delta$.
\end{clm}

\begin{proof}[Proof of Claim~\ref{samebridge:clm7}]
Suppose $(a,b_1),(a,b_2) \in R$.  
Let $u:=\pi_{\m a}(a)$, $y_i:=\pi_{\m b}(b_i)$, and
$u_i' = \delta^{-1}(y_i)$ for $i=1,2$. 
Then $(a,u,b_i,y_i) \in T$ and $(u,u_i') \in \theta$ for $i=1,2$.
Hence $(u_1',u_2') \in \theta$ by 
Claim~\ref{samebridge:clm5}\eqref{sb-clm5:it3},
so $d(u_1',u_1',u_2')=u_2'$ by 
Claim~\ref{samebridge:clm5}\eqref{sb-clm5:it1}, so
$d(y_1,y_1,y_2)=y_2$ as $\delta$ is an isomorphism.
Also note that $(a,u_1')$ and $(u,u_1')$ are Maltsev pairs by
Claim~\ref{samebridge:clm5}\eqref{sb-clm5:it0}, so
$d(u_1',a,a)=d(u_1',u,u)=u_1'$.  Finally, from $(u_1',y_1)\in \graph(\delta)
\subseteq \tr(T)$ we get $(u_1',u_1',y_1,y_1)\in T$.  Thus applying
$d$ to 
\[
(u_1',u_1',y_1,y_1),\quad (a,u,b_1,y_1), \quad (a,u,b_2,y_2),
\]
we get
\[
(u_1',u_1',d(y_1,b_1,b_2),y_2) \in T.
\]
By \ref{pbridge:it2} we get $d(y_1,b_1,b_2)=y_2$
and so $(u_1',y_2) \in \tr(T) \cap (S_{\m a}\times S_{\m b}) = \graph(\delta)$
by \eqref{eq:graphdelta},
so $y_2=y_1$.  Thus $y_1,b_1,b_2$ belong to a common $\montwo$-class and
$d(y_1,b_1,b_2)=y_1$, which forces $b_1=b_2$.

This and Claim~\ref{samebridge:clm4} prove 
that $R$ is the graph of a surjective homomorphism $g:\m a\ra \m b$.
A symmetrical argument proves that $g$ is injective.  $g$ extends $\delta$
by Claim~\ref{samebridge:clm2}\eqref{sb-clm2:it1}, which completes
the proof of Claim~\ref{samebridge:clm7}.
\end{proof}

\begin{clm} \label{samebridge:clm8}
$R\subseteq \tr(T)$.
\end{clm}

\begin{proof}[Proof of Claim~\ref{samebridge:clm8}]
Let $(a,b) \in R$, $u = \pi_{\m a}(a)$ and $y=\pi_{\m b}(b)$.
Then $(a,u,b,y) \in T$.  Since $g$ is an isomorphism,
we have $g(\mu)=\montwo$, and since $g(a)=b$ and $g(u)=\delta(u)$,
we get $(b,\delta(u))=(g(a),g(u)) \in \montwo$.
Since $\delta(u),y\in b/\montwo \cap S_{\m b}$, we get $\delta(u)=y$,
so $(u,y) \in \graph(\delta)\subseteq \tr(T)$.   As $T=\mu\circ T\circ \montwo$,
we get $(a,b) \in \tr(T)$.
\end{proof}

Now we can finish the proof of Theorem~\ref{thm:samebridge}.  
Pick any $a \in A\setminus S_{\m a}$.
Let $b=g(a)$, so $(a,b) \in R$.  
Also let $u = \pi_{\m a}(a)$ and $y=\pi_{\m b}(b)$.
Then $(a,u,b,y)\in T$,
$a\ne u$, and $(a,b) \in \tr(T)$ by Claim~\ref{samebridge:clm8}.
Thus $(a_1,a_2,b_1,b_2):=(a,u,b,y)$ satisfies the hypotheses
of Claim~\ref{samebridge:clm1}, and hence by that Claim there exists a 
bridge $T_1$ from $(\m a,0,\mu)$ to
$(\m b,0,\montwo)$ satisfying $T_1\subseteq T$,
$\tr(T_1)=\tr(T)$ and \ref{pbridge:it3}, as required.
\end{proof}

As a consequence, \rooted\ bridges between meet-irreducible congruences
encode similarity between the respective subdirectly irreducible quotients.

\begin{cor} \label{cor:zhuk-equiv}
Suppose $\m a,\m b$ are finite algebras in a locally finite Taylor variety,
$\rho \in \Con(\m a)$ and $\sigma \in \Con(\m b)$ where both $\rho$ and $\sigma$
are meet-irreducible, and $\coverrho,\coversigma$ are their respective unique
upper covers.
Then the following are equivalent:
\begin{enumerate}
\item \label{zhuk-equiv-it1}
There exists a \rooted\ bridge from $(\m a,\rho)$ to $(\m b,\sigma)$.
\item \label{zhuk-equiv-it2b}
There exists a bridge from $(\m a,\rho,\coverrho)$ to $(\m b,\sigma,\coversigma)$
which satisfies \ref{pbridge:it3}.
\item \label{zhuk-equiv-it3}
$\m a/\rho$ and $\m b/\sigma$ are similar; i.e., $\m a/\rho \sim \m b/\sigma$.
\end{enumerate}
\end{cor}

\begin{proof}
Let $\barma = \m a/\rho$ and $\barmb = \m b/\sigma$.
Let $\mu$ and $\montwo$ be the monoliths of $\barma$ and $\barmb$ respectively.
By the discussion following Definition~\ref{df:pbridge},
\eqref{zhuk-equiv-it2b} is equivalent to
\begin{enumerate}[label=(\arabic*$'$)]
\setcounter{enumi}{1}
\item \label{zhuk-equiv-it3'}
There exists a bridge from $(\overline{\m a},0_{\barA},\mu)$ to 
$(\overline{\m b},0_{\barB},\montwo)$ which satisfies \ref{pbridge:it3},
\end{enumerate}
which in turn is equivalent to 
\begin{enumerate}[label=(\arabic*$''$)]
\setcounter{enumi}{1}
\item \label{zhuk-equiv-it3''}
There exists a \proper\ bridge from $\barma$ to $\barmb$.
\end{enumerate}
Thus \eqref{zhuk-equiv-it2b}\Iff\eqref{zhuk-equiv-it3} by
Theorem~\ref{thm:persp}, and clearly
\eqref{zhuk-equiv-it2b}\Implies \eqref{zhuk-equiv-it1}.

It remains to prove \eqref{zhuk-equiv-it1}\Implies\eqref{zhuk-equiv-it2b}.
Assume that $T$ is a \rooted\ bridge from $(\m a,\rho)$ to $(\m b,\sigma)$.
If either $\coverrho/\rho$ or $\coversigma/\sigma$ is nonabelian, then both are nonabelian 
and $\barma \cong \barmb$ by Lemma~\ref{lm:sametype}\eqref{sametype:it1}; 
hence $\barma \sim \barmb$, proving \eqref{zhuk-equiv-it3} and hence 
\eqref{zhuk-equiv-it2b} in this case.

In the remaining case, both $\coverrho/\rho$ and $\coversigma/\sigma$
are abelian.  Then $\Cov(\rho)=\{\coverrho\}$ and
$\Cov(\sigma)=\{\coversigma\}$ by Lemma~\ref{lm:opt2}\eqref{opt2:it3}.  Since $T$ is
\rooted, it contains a bridge from $(\m a,\rho,\coverrho)$ to $(\m b,\sigma,\coversigma)$.
Then by Theorem~\ref{thm:samebridge}, there exists a bridge $T'$ from 
$(\m a,\rho,\coverrho)$
to $(\m b,\sigma,\coversigma)$ which satisfies \ref{pbridge:it3}, proving
\eqref{zhuk-equiv-it2b} in this case as well.
\end{proof}

As a special case, we get the following characterization of the ``there exists
a bridge" relation between irreducible congruences of finite Taylor algebras.

\begin{cor} \label{cor:zhuk-equiv2}
Suppose $\m a,\m b$ are finite algebras in a locally finite Taylor variety,
and $\rho \in \Con(\m a)$ and $\sigma \in \Con(\m b)$ where both $\rho$ and $\sigma$
are irreducible.
Then the following are equivalent:
\begin{enumerate}
\item 
There exists a bridge from $(\m a,\rho)$ to $(\m b,\sigma)$.
\item 
$\m a/\rho \sim \m b/\sigma$.
\end{enumerate}
\end{cor}

\begin{proof}
Every bridge between irreducible congruences is \rooted.
\end{proof}

\bibliography{similar-bib}
\end{document}